\documentclass[10pt]{article}

\usepackage{microtype}
\usepackage{graphicx}
\usepackage{subcaption}
\usepackage{booktabs}

\pdfoutput=1

\usepackage{amsmath,amsbsy,amsgen,amscd,amssymb,amsthm,amsfonts,stmaryrd}
\usepackage{mathtools}

\usepackage{bm}
\usepackage{enumerate}

\usepackage{microtype}      

\usepackage[colorlinks,citecolor=blue]{hyperref}
\usepackage{cleveref}
\usepackage{url}            

\usepackage[usenames,dvipsnames,svgnames]{xcolor}
\usepackage{graphicx}

\graphicspath{{art/}}

\definecolor{dark-gray}{gray}{0.3}
\definecolor{dkgray}{rgb}{.4,.4,.4}
\definecolor{dkblue}{rgb}{0,0,.5}
\definecolor{medblue}{rgb}{0,0,.75}
\definecolor{rust}{rgb}{0.7,0.1,0.1}
\definecolor{drust}{rgb}{0.5,0.1,0.1}

\hypersetup{urlcolor=drust}
\hypersetup{citecolor=blue}
\hypersetup{linkcolor=rust}

\newtheorem{theorem}{Theorem}
\newtheorem{lemma}[theorem]{Lemma}
\newtheorem{corollary}[theorem]{Corollary}

\numberwithin{equation}{section}
\numberwithin{theorem}{section}

\newcommand{\R}{\mathbb{R}}

\renewcommand{\phi}{\varphi}

\newcommand{\proj}{\operatorname{proj}}

\DeclareFontFamily{OT1}{pzc}{}
\DeclareFontShape{OT1}{pzc}{m}{it}{<-> s * [1.200] pzcmi7t}{}
\DeclareMathAlphabet{\mathpzc}{OT1}{pzc}{m}{it}

\DeclareMathOperator{\prox}{prox}

\DeclareMathOperator*{\argmin}{arg\,min}

\newtheorem{assumption}{Assumption \!\!}

\usepackage{xcolor,colortbl}

\makeatletter
\newtheorem*{rep@theorem}{\rep@title}
\newcommand{\newreptheorem}[2]{%
	\newenvironment{rep#1}[1]{%
		\def\rep@title{#2 \ref{##1}}%
		\begin{rep@theorem}}%
		{\end{rep@theorem}}}
\makeatother

\newreptheorem{theorem}{Theorem}
\newreptheorem{lemma}{Lemma}

\theoremstyle{definition}

\usepackage{lscape}
\usepackage{makecell}

\usepackage{natbib}
\usepackage{algorithm}
\usepackage{algorithmic}
\usepackage[margin=1in]{geometry}

\crefname{theorem}{Theorem}{Theorems}
\Crefname{theorem}{Theorem}{Theorems}
\crefname{proposition}{Proposition}{Propositions}
\Crefname{proposition}{Proposition}{Propositions}
\crefname{lemma}{Lemma}{Lemmas}
\Crefname{lemma}{Lemma}{Lemmas}
\crefname{corollary}{Corollary}{Corollaries}
\Crefname{corollary}{Corollary}{Corollaries}
\crefname{definition}{Definition}{Definitions}
\Crefname{definition}{Definition}{Definitions}
\crefname{assumption}{Assumption}{Assumptions}
\Crefname{assumption}{Assumption}{Assumptions}
\crefname{remark}{Remark}{Remarks}
\Crefname{remark}{Remark}{Remarks}

\crefname{appendix}{Appendix}{Appendices}
\Crefname{appendix}{Appendix}{Appendices}
\crefname{appsubsection}{Appendix}{Appendices}
\Crefname{appsubsection}{Appendix}{Appendices}

\newcommand{\E}{\mathbb{E}}

\newcommand{\vx}{\mathbf{x}}
\newcommand{\vy}{\mathbf{y}}
\newcommand{\vz}{\mathbf{z}}
\newcommand{\vu}{\mathbf{u}}

\newcommand{\va}{\mathbf{a}}
\newcommand{\vb}{\mathbf{b}}
\newcommand{\mX}{\mathbf{X}}

\usepackage[colorinlistoftodos,prependcaption]{todonotes}

\title{Convergence Rate of the Last Iterate of Stochastic Proximal Algorithms}
\author{Kevin Kurian Thomas Vaidyan\footnote{Department of Computer Science, University of British Columbia. \url{kevinktv@cs.ubc.ca}} \and Michael P. Friedlander$^{1,}$\footnote{Department of Computer Science, University of British Columbia. \url{michael.friedlander@ubc.ca}}\and Ahmet Alacaoglu$^{1,}$\footnote{Department of Mathematics, University of British Columbia. \url{ahmet.alacaoglu@ubc.ca}}}
\date{}
\begin{document}

\maketitle
\begin{abstract}
    We analyze two classical algorithms for solving additively composite convex optimization problems where the objective is the sum of a smooth term and a nonsmooth regularizer. The first algorithm is the proximal stochastic gradient method for a single regularizer; the second is the randomized incremental proximal method, which uses the proximal operator of a randomly selected function when the regularizer is given as the sum of many nonsmooth functions. We focus on relaxing the bounded variance assumption that is common, yet stringent, for getting last iterate convergence rates. We prove the $\widetilde{O}(1/\sqrt{T})$ rate of convergence for the last iterate of both algorithms under componentwise convexity and smoothness, which is optimal up to log terms. Our results apply directly to graph-guided regularizers that arise in multi-task and federated learning, where the regularizer decomposes as a sum over edges of a collaboration graph.
\end{abstract}

\section{Introduction}
\footnotetext[1]{Joint last authors.}

Composite optimization problems with structured regularizers arise throughout machine learning. In multi-task learning, for example, related tasks share statistical structure that can be harnessed through \emph{graph-guided regularizers} \citep{kim2009statistical}: a collaboration graph connects tasks with shared characteristics, and the regularizer encourages their parameters to agree. This formulation similarly manifests in federated optimization, where agents learn related models without centralizing data, and in sensor networks that aggregate spatially correlated measurements. \citep{smith2017federated}.

A canonical instance is the \emph{network Lasso} \citep{hallac2015network}, where each node $i$ maintains parameters $\vx_i$ and the regularizer penalizes disagreement between neighbors:
\[
g(\vx) = \sum_{(i,j)\in\mathcal{E}} w_{ij} \|\vx_i - \vx_j\|_p.
\]
The regularizer naturally decomposes as a sum over edges, $g = \sum_{i} g_i$, where each $g_i$ involves only the variables at connected nodes.

More generally, consider the regularized optimization problem
\begin{equation}\label{eq:problem}
    \min_{\vx\in\mathbb{R}^n}\ h(\vx) := f(\vx) + g(\vx), 
\end{equation}
where $f\colon\mathbb{R}^n\to\mathbb{R}$ is convex and differentiable, and $g\colon\mathbb{R}^n\to\mathbb{R}\cup\{+\infty\}$ is convex. We assume access to a stochastic gradient oracle that returns unbiased estimates $\nabla f_i$ that satisfy 
\begin{equation*}
    \mathbb{E}[\nabla f_i(\vx)] = \nabla f(\vx) \quad \text{for all } \vx.
\end{equation*}
A standard assumption for analyzing stochastic gradient descent (SGD) requires the gradient estimates to have a  uniformly bounded variance. This assumption is overly restrictive and fails even for simple problems; see \cref{eq: stoc_grad_lasso} for an example.

\citet{garrigos2025last} and \citet{attia2025fastlastiterateconvergencesgd}
showed that the bounded variance assumption can be avoided for last-iterate analyses of SGD
in unconstrained optimization by instead requiring that
\begin{equation}\label{eq:componentwise-smooth}
    \text{each $f_i$ is convex and $L$-smooth.}
\end{equation}
Under this condition, the last iterate of SGD achieves the optimal (up-to-$\log$)
$\widetilde{O}(1/\sqrt{T})$ rate. The extension to
\emph{regularized} problems, however, does not follow directly:
the proximal operator is nonlinear, so the iterates no longer admit
the closed-form expansion that these analyses utilize.

The workhorse method for regularized problems in machine learning
is the proximal SGD \citep{duchi2009efficient}, which iterates as
\begin{equation*}\label{eq:intro-spgd}
    \vx_{t+1} = \prox_{\tau g}(\vx_t - \tau \nabla f_{i_t}(\vx_t)), \tag{SPGD}
\end{equation*}
where $\prox_{g}(\vx) := \argmin_{\vz} \{ g(\vz) + \tfrac{1}{2}\|\vx-\vz\|^2 \}$ and $i_t$ is selected i.i.d. at every iteration with respect to a fixed distribution $D$.
(More generally, the step size $\tau$ may vary with each iteration,
but our analysis focuses on constant step sizes.)
Our purpose is to extend the last-iterate guarantees of
\citet{garrigos2025last,attia2025fastlastiterateconvergencesgd} to this setting.

\textbf{A classical example. } Let us start with perhaps the most standard, textbook example of a regularized problem, Lasso, or linear least squares regression with $\ell_1$ regularization, given as
\begin{equation*}
\min_{\vx\in\mathbb{R}^n}\frac{1}{2N} \sum_{i=1}^N (\langle \va_i, \vx\rangle - b_i)^2 + \lambda \| \vx\|_1,
\end{equation*}
where $\va_i \in \mathbb{R}^n$ and $b_i \in \mathbb{R}$. Of course, there are many methods developed for this problem, including variance reduced algorithms \citep{gower2020variance}. However, when the number of data points $N$ is extremely large, SGD-based methods are the main choices since they have rates of convergence not depending on $N$, allowing arbitrarily large $N$. 

Here, the typical choice of a stochastic gradient is
\begin{equation}\label{eq: stoc_grad_lasso}
\nabla f_i(\vx) = \va_i(\langle \va_i, \vx\rangle - b_i).
\end{equation}
Even for this problem, classical assumptions such as the bounded variance (or other distributional assumptions) do not necessarily hold since the quantity $\nabla f_i(\vx)  -\nabla f(\vx)$ is linear in $\vx$. Hence, the existing results showing the last-iterate convergence rate of proximal SGD from \citet{liu2024revisitinglastiterateconvergencestochastic} do not apply. Moreover, due to the existence of the $\ell_1$ regularizer, existing results for unconstrained SGD without bounded variance from \citet{garrigos2025last,attia2025fastlastiterateconvergencesgd} also do not apply. 

That is, the literature is currently missing the theoretical guarantees for the most standard setup for solving this classical problem in the large-scale regime: proximal SGD with last-iterate as the output. Our results address this gap.

Our main contributions are as follows:

\begin{itemize}
    \item We prove that the last iterate of \emph{proximal SGD} converges at the
    optimal (up-to-$\log$) rate $\widetilde{O}(1/\sqrt{T})$ under componentwise convexity and
    smoothness of $f_i$, with second moment of stochastic gradients bounded only at the solution
    (\cref{BIGTHEOREM_main}). As a special case, we obtain the same rate for
    projected SGD.
    
    \item When the regularizer decomposes additively as $g(\vx) = \sum_{i=1}^m g_i(\vx)$,
    the proximal map $\prox_g$ may be intractable even when each $\prox_{g_i}$
    admits closed form. We prove the $\widetilde{O}(1/\sqrt{T})$ last-iterate
    rate for the \emph{randomized incremental proximal method} \citep{bertsekas2011incremental}
    \begin{equation*}
        \vx_{t+1} = \prox_{\tau m g_{j_t}}(\vx_t - \tau \nabla f_{i_t}(\vx_t)), \label{eq:intro-incrementalpgd} \tag{RIPM}
    \end{equation*}
    where $j_t$ is selected uniformly at random and $i_t$ is selected as \eqref{eq:intro-spgd} (\cref{BIGTHEOREM2_main}).    This implies optimal (up-to-$\log$) rates for stochastic proximal point methods.
    
    \item For graph-guided regularizers, such as network Lasso and multi-task learning with task similarity constraints, our results lead to last-iterate convergence for the BlockProx algorithm \citep{LinKuangAlacaogluFriedlander2025}. We thus establish optimal rates where only ergodic convergence was previously known without bounded variance.
    
    \item We provide numerical experiments confirming that the last iterate outperforms
    averaged iterates in practice for proximal SGD.
\end{itemize}

\paragraph{A note on presentation. } Although \eqref{eq:intro-incrementalpgd} generalizes \eqref{eq:intro-spgd}, we analyze the simpler algorithm first in \cref{sec: proxsgd} and treat the incremental setting in \cref{sec: rand_inc_prox}. This ordering serves three purposes: 
\emph{(i)} Handling randomly selected proximal operators in \eqref{eq:intro-incrementalpgd} alongside stochastic gradients requires additional analytical machinery, so we focus first on the case where only gradients are stochastic (the algorithm \eqref{eq:intro-spgd}) so that we can delineate the particular difficulties in the analysis and our solutions. \emph{(ii)} Projected and proximal SGD algorithms are interesting in their own right as classical tools for solving important classes of regularized  and constrained optimization problems. \emph{(iii)} The generality of \eqref{eq:intro-incrementalpgd} requires further assumptions for the problem template in \cref{sec: rand_inc_prox} and hence the results are not directly comparable.

\section{Problem setting and main assumptions}\label{sec: prob_set}

The standard bounded variance assumption in stochastic optimization requires the following inequality to hold uniformly over all $\vx$: $\E[\|\nabla f_i(\vx) - \nabla f(\vx)\|^2] \leq \sigma^2$. This overly restrictive condition fails even for unconstrained linear least-squares, see for example \eqref{eq: stoc_grad_lasso}. For unconstrained smooth problems (i.e., $g = 0$), recent work has shown that it suffices to assume finite second moments at the solution:
\begin{equation}\label{eq:second-moment-solution}
    \E\|\nabla f_i(\vx^*)\|^2 = \sigma_{*}^2 < \infty
    \quad\text{for some } \vx^* \in \argmin f.
\end{equation}
When multiple minimizers exist, this bound holds for every minimizer \citep[Lemma~8.23]{garrigos2024handbook}. For the composite problem~\eqref{eq:problem} with $g \neq 0$, we impose the analogous condition at minimizers of $h = f + g$. The analysis, however, requires new tools because the proximal operator introduces additional structure not present in the unconstrained setting.

We consider two settings for problem~\eqref{eq:problem}:
\begin{enumerate}
    \item \emph{Proximal SGD.} The smooth component $f$ is an expectation over a distribution $D$ from which we can draw i.i.d. samples: 
    \begin{equation}\label{eq: fdef_exp}
    f(\vx) = \E_{i \sim D}[f_i(\vx)],
    \end{equation}
     and we have access to the full proximal operator $\prox_{\tau g}$.
    \item \emph{Incremental proximal.} In addition, the regularizer decomposes as $g(\vx) = \sum_{j=1}^m g_j(\vx)$, and we access only proximal operators of individual components $g_j$.
\end{enumerate}

\begin{assumption}[Standing assumptions]\label{assumptions}
    \hfill
    \begin{enumerate}
        \item The solution set $X^* = \argmin_{\vx} h(\vx)$ is nonempty. We write $\vx^*$ for an arbitrary element of $X^*$, $h^* = h(\vx^*)$ for the optimal value, and 
        \begin{equation*}
            D_*^2 = \min_{\vx^* \in X^*}\|\vx^* - \vx_0\|^2,
        \end{equation*}
        for the squared distance from the initial point to the solution set. Let also $\mathbb{E}\|\nabla f_i(\vx^*)\|^2 = \sigma_*^2 < \infty$ for some $\vx^* \in \arg\min h$.
        
        \item Every component function $f_i\colon \R^n \to \R$ is convex and differentiable with $L$-Lipschitz gradient.
        
        \item The regularizer $g\colon \R^n \to \R \cup \{+\infty\}$ is proper, convex, and lower semicontinuous.
    \end{enumerate}
\end{assumption}

For the incremental proximal setting, we impose additional structure on the regularizer.

\begin{assumption}[Decomposable regularizer]\label{assumptions2_main}
    \hfill
    \begin{enumerate}
        \item Each $g_j\colon \R^n \to \R$ is proper, convex, and lower semicontinuous.
        \item Each $g_j$ is $L_g$-Lipschitz: $|g_j(\vx) - g_j(\vy)| \leq L_g\|\vx - \vy\|$ for all $\vx, \vy \in \R^n$.
    \end{enumerate}
\end{assumption}

\noindent The Lipschitz condition on $g_j$ controls the variance introduced by proximal stochasticity. This requirement is common in methods that access nonsmooth terms randomly, in the non-strongly convex case, see for example \citep{asi2019stochastic,bertsekas2011incremental,cai2025last}. We will share more insights regarding the need for this condition in Section \ref{sec: main_ideas_incr}.

\paragraph{Notation. } Given iterates $(\vx_t)$, we denote as $\mathcal{F}(\vx_0, \dots \vx_t)$ the $\sigma$-algebra generated by $\vx_0, \dots \vx_t$. The standard notation $\E_t[\cdot]$ is used for the expectation conditioned on $\mathcal{F}(\vx_0, \dots \vx_t)$, that is, $\mathbb{E}[\cdot|\mathcal{F}_t] = \mathbb{E}_t[\cdot]$.

\subsection{Variance control via co-coercivity}\label{sec:starting-point}

This section presents the mechanism that relaxes the bounded variance assumption for unconstrained SGD, and explains why the proximal setting requires additional analytical tools.
The mechanism relies on co-coercivity: convexity and $L$-smoothness of each $f_i$ imply \citep[Thm.~2.1.5]{nesterov2018lectures}
\begin{equation}\label{eq:cocoercivity}
    \begin{aligned}
        \|\nabla f_i(\vx) - \nabla f_i(\vy)\|^2 
        \leq 2L\bigl(f_i(\vx) - f_i(\vy) - \langle \nabla f_i(\vy), \vx - \vy \rangle\bigr).
    \end{aligned}
\end{equation}
The right-hand side involves only function values, and thus bounds the stochastic gradient norm $\|\nabla f_i(\vx_t)\|^2$ without additional assumptions.

Consider standard SGD for unconstrained minimization, where $g \equiv 0$ and let $\vx^*\in\argmin f$:
\[
\vx_{t+1} = \vx_t - \tau \nabla f_{i_t}(\vx_t),
\]
where $i_t\sim D$ is selected i.i.d., see \eqref{eq: fdef_exp}.

Expanding the squared distance to a reference point $\vz$ gives
\begin{align}\label{eq:sgd-expansion}
    \|\vx_{t+1} - \vz\|^2
    = \|\vx_t - \vz\|^2
    - 2\tau \langle \nabla f_{i_t}(\vx_t), \vx_t - \vz \rangle + \tau^2 \|\nabla f_{i_t}(\vx_t)\|^2.
\end{align}
Taking expectation over the choice of $i_t$ and applying convexity of each $f_{i}$ to the inner product yields
\begin{align}\label{eq:sgd-expected}
    \mathbb{E}\|\vx_{t+1} - \vz\|^2
    &\leq \mathbb{E}\|\vx_t - \vz\|^2
    \underbrace{- 2\tau (f(\vx_t) - f(\vz))}_{\text{descent term}} 
    + \underbrace{\tau^2 \mathbb{E}\|\nabla f_i(\vx_t)\|^2}_{\text{variance term}}.
\end{align}
The descent term is negative whenever $f(\vx_t) > f(\vz)$ and thus drives convergence. The variance term, however, is positive, and without an \emph{a priori} bound, it prevents meaningful control of the distance to $\vz$.

We now apply the co-coercivity inequality~\eqref{eq:cocoercivity} with $i=i_t$, $\vx=\vx_t$, $\vy = \vx^*$, and use Young's inequality to deduce
\begin{align}\label{eq: var_bd5}
    \|\nabla f_{i_t}(\vx_t)\|^2 
    \leq 2\|\nabla f_{i_t}(\vx^*)\|^2
    + 4L \bigl( f_{i_t}(\vx_t) - f_{i_t}(\vx^*)
    - \langle \nabla f_{i_t}(\vx^*),\, \vx_t - \vx^* \rangle \bigr).
\end{align}
Taking expectation over ${i_t}$, the inner product vanishes because $\E_{i_t}[\nabla f_{i_t}(\vx^*)] = \nabla f(\vx^*) = 0$. Multiplying \eqref{eq: var_bd5} by $\tau^2$ and using~\eqref{eq:second-moment-solution}  yields the variance bound
\begin{equation}\label{eq:variance-bound}
    \tau^2\mathbb{E}_t\|\nabla f_{i_t}(\vx_t)\|^2 \leq 2\tau^2\sigma_*^2 + 4\tau^2 L(f(\vx_t) - f(\vx^*)).
\end{equation}
The term $f(\vx_t) - f(\vx^\star)$ on the right-hand side has the same form as the descent term in \eqref{eq:sgd-expected}. For $\tau \leq 1/(2L)$ and $\vz = \vx^*$ for some $\vx^*$, the variance contribution $4\tau^2 L(f(\vx_t) - f(\vx^*))$ is at most $2\tau(f(\vx_t) - f(\vx^*))$, and thus absorbed by the descent term. The residual $2\tau^2 \sigma_*^2$ remains bounded and determines the final convergence rate. This would give the guarantee on the average of the iterates of the algorithm. 

For last iterate guarantees, one needs to plug in $\vz=\vz_t$ for a special choice of $\vz_t$ (cf. \Cref{subsubsec:last-iterate-reduction}) and apply a more involved analysis \citep{garrigos2025last}. Yet, the mechanism explained above remains key to go beyond the bounded variance or bounded gradient assumptions.

This argument relies on two properties that fail when $g \not\equiv 0$.
First, the expansion~\eqref{eq:sgd-expansion} relies on the linear structure of SGD. The proximal operator is nonlinear, so this identity fails. We replace it with the three-point identity and the prox-inequality (\Cref{subsubsec:one-iteration}).
Second, the cancellation $\E[\langle \nabla f_i(\vx^\star), \vx_t - \vx^\star \rangle] = 0$ that is used to get to \eqref{eq:variance-bound} requires $\nabla f(\vx^\star) = 0$. With a nontrivial regularizer, denoting $\vx^* \in \argmin h$, the optimality condition $0 \in \nabla f(\vx^*) + \partial g(\vx^*)$ allows $\nabla f(\vx^*) \neq 0$.
Overcoming these obstacles while preserving last-iterate guarantees constitutes the main technical contribution of this work. 

\subsection{Related Work}

\begin{table*}[t]
    \centering
    \footnotesize
    \begin{tabular}{l c c c c c c}
        \toprule
        &
        \textbf{\makecell{Oracle\\for $g$}} &
        \textbf{\makecell{Unbounded \\ variance and \\ no distribution \\ assumption?}} &
        \textbf{\makecell{Handles\\prox?}} &
        \textbf{\makecell{Last\\iter.}} &
        \textbf{\makecell{Bound indep.\\\# smooth funcs}} &
        \textbf{\makecell{Stochastic \\prox} }\\
        \midrule
        Folklore & $\prox_g$ & $\times$ & $\checkmark$ & $\times$ & $\checkmark$ & $\times$ \\[1.5mm]
        \citet{khaled2023unified} & $\prox_g$ & $\checkmark$ & $\checkmark$ & $\times$ & $\checkmark$ & $\times$ \\[1.5mm]
        \citet{garrigos2025last, attia2025fastlastiterateconvergencesgd} & N/A & $\checkmark$ & $\times$ & $\checkmark$ & $\checkmark$ & $\times$ \\[1.5mm]
        \citet{liu2024revisitinglastiterateconvergencestochastic} & $\prox_g$ & $\times$ & $\checkmark$ & $\checkmark$ & $\checkmark$ & $\times$ \\[1.5mm]
        \citet{bertsekas2011incremental} & $\prox_{g_i}$ & $\times$ & $\checkmark$ & $\checkmark$ & $\checkmark$ & $\checkmark$ \\[1.5mm]
        Section \ref{sec: proxsgd} & $\prox_g$ & $\checkmark$ & $\checkmark$ & $\checkmark$ & $\checkmark$ & $\times$ \\[1.5mm]
        Section \ref{sec: rand_inc_prox} & $\prox_{g_i}$ & $\checkmark$ & $\checkmark$ & $\checkmark$ & $\checkmark$ & $\checkmark$ \\
        \bottomrule
    \end{tabular}
    \caption{Comparison of relevant results}
    \label{tab:comparison}
\end{table*}
Since we consider two sets of methods, based on stochastic gradient descent or stochastic proximal point methods, we divide the comparison to two sections. A summary of the comparison of our results to the most related ones in the literature is provided in \Cref{tab:comparison}.

\subsubsection{SGD world}\label{sec: sgd_world}

Most convergence analyses for SGD assume bounded variance \citep{nemirovski2009robust,needell2014stochastic,bottou2018optimization}. Several works relax this assumption \citep{poljak1973pseudogradient,gladyshev1965stochastic,khaled2022better,khaled2023unified,neu2024dealing,alacaoglu2025towards,moulines2011non}, though most focus on the average iterate rather than the last iterate, and some are restricted to unconstrained problems.

For the last iterate, most results assume bounded variance, bounded gradients, or other distributional assumptions on the gradient noise \cite{shamir2013stochastic,harvey2019tight,orabona2020last,liu2024revisitinglastiterateconvergencestochastic}; some extend to composite problems. The influential work by \citet{moulines2011non} established a suboptimal $O(T^{-1/3})$ rate for the last iterate. Two recent works \cite{attia2025fastlastiterateconvergencesgd,garrigos2025last} achieve the near-optimal $\widetilde{O}(T^{-1/2})$ rate without bounded variance, under \eqref{eq:componentwise-smooth}.

A key feature of these results is \emph{N-independence}: they apply even when the number of component functions $N$ in the sum $f(\vx) = (1/N)\sum_{i=1}^N f_i(\vx)$ is infinite, with convergence bounds that do not depend on $N$.

A complementary line of work focuses on finite $N$ with non-i.i.d.\ sampling, such as without replacement or cyclic selection rather than with replacement. Two recent works establish last-iterate guarantees: \citet{cai2025last} study incremental gradient and incremental proximal-point methods for unconstrained problems, where components are selected cyclically; \citet{liu2025improved} address the composite problem but require $\prox_g$.

\emph{The key distinction is the following:} these convergence bounds depend polynomially on $N$, hence require a finite number of component functions; our bounds are $N$-independent. That is, the results of \citet{cai2025last} and \citet{liu2025improved} require multiple passes over the data for making progress, whereas our bounds show progress with a single pass. On the other hand, if one is to focus on this different setting with finite $N$, these results may give a better complexity. Moreover, neither work considers stochastic sampling of regularizers (using $\prox_{g_i}$ when $g(\vx)=\sum_{i=1}^m g_i(\vx)$), which is the focus of Sec.~\ref{sec: rand_inc_prox}.

\subsubsection{Proximal point world}
We now consider the more general setting where the algorithm accesses unbiased samples of both the smooth part $f$ and the nonsmooth part $g$. In particular, given a finite-sum problem 
\begin{equation}
    \min_{\vx} \tfrac{1}{N}\textstyle\sum_{i=1}^N f_i(\vx) + g_i(\vx) ,
\end{equation}
stochastic proximal point (SPP) methods iterate as
\begin{align*}
    \vx_{t+1} = \prox_{\tau(f_i+g_i)}(\vx_t),
\end{align*}
for $i\in\{1,\dots,N\}$ selected uniformly at random.

The proximal operator $\prox_{f_i + g_i}$ is often intractable even when $\prox_{f_i}$ and $\prox_{g_i}$ admit closed forms. Consequently, most works on stochastic proximal point methods assume $g_i = 0$ \cite{asi2019stochastic,bianchi2016ergodic}. A more practical approach treats the smooth and nonsmooth terms separately, and apply the gradient of $f_i$ followed by the proximal operator of $g_i$, see \eqref{eq:intro-incrementalpgd}. \citet{bertsekas2011incremental} analyzed this strategy  in order to establish asymptotic convergence of the sum $f(\vx_k) + g(\vx_k)$ but without nonasymptotic rates for the last iterate.

Recent work by \citet{tovmasyan2025revisiting} and \citet{condat2025stochastic} provides nonasymptotic guarantees for related methods, but under strong convexity assumptions that exclude many regularized problems of practical interest.

Variance reduction techniques have also been applied to stochastic proximal point methods. \citet{traore2024variance} study finite sum problems and establish an ergodic rate of $\mathcal{O}(1/T)$ for convex smooth objectives. However, their focus is on the averaged iterate rather than the last iterate, and they do not consider the case where $g$ is accessed via $\prox_{g_i}$.

\section{Proximal SGD}\label{sec: proxsgd}

This section establishes the last-iterate convergence rate for proximal SGD \eqref{eq:intro-spgd}. We state the main theorem, outline the proof strategy, and specialize the result to projected SGD. Full proofs appear in \cref{app: spgd}.
\subsection{Statement of the result}\label{sec: proxsgd_statement}
The following theorem gives the last-iterate convergence rate for \eqref{eq:intro-spgd}. The proof appears in \cref{subsec: proof_main_thm_spgd}.
\begin{theorem}[Last-iterate convergence]\label{BIGTHEOREM_main}
    Let \cref{assumptions} hold.  In \eqref{eq:intro-spgd}, let $\tau = 1/(3 L \sqrt{T})$. Then
    \begin{align*}
        \E\left[h(\vx_{T+1}) - h^* \right] &\leq \frac{9}{\sqrt{T}}\Big[ LD_*^2+\frac{1}{\sqrt{T}}(h(\vx_0) - h^*) +\frac{\sigma_*^2}{L}\Big(\frac{1}{T}+4\ln(T+1)\Big) \Big]\!.
    \end{align*}
\end{theorem}
The $O(\ln(T+1)/\sqrt{T})$ rate matches known last-iterate guarantees for unconstrained SGD \citep{garrigos2025last,attia2025fastlastiterateconvergencesgd}. Compared to the optimal $O(1/\sqrt{T})$ rate that is achieved by averaged iterates or with non-standard step sizes under bounded variance \citep{jain2019making}, our bound incurs an extra logarithmic term of $\ln (T+1)$.

When the objective has finite-sum structure $f = \frac{1}{N}\sum_{i=1}^N f_i$, our bound is independent of $N$, which allows $N$ to be arbitrarily large and is a key benefit of SGD, see \Cref{sec: sgd_world} for this discussion. The bound depends only on the smoothness constant $L$, the variance at the solution $\sigma_*^2$, and the squared initial distance $D_*^2$; see \cref{sec: prob_set} for precise definitions.

\subsection{Proof setup}\label{subsec: main_ideas_psgd}
The proof of \Cref{BIGTHEOREM_main} proceeds in two stages: a one-iteration analysis that bounds the expected progress per step, followed by a last-iterate reduction that telescopes these bounds into a final iterate guarantee.

\subsubsection{One-iteration analysis}\label{subsubsec:one-iteration}
The essential intermediate result is a one-iteration bound for proximal SGD that generalizes \citet[Lemma 4.2]{garrigos2025last} to the proximal setting. The full proof appears in \Cref{subsec: one-it-spgd}.
\begin{lemma}[Per-iteration progress] \label{proxobjectivecomb_main}
    Let \Cref{assumptions} hold. In \eqref{eq:intro-spgd}, pick $\tau$ such that $\tau < 1/(2L)$.
    Then for all $t\in[0, T]$ and $\vz_t$ in $\mathcal{F}(\vx_0, ..., \vx_t)$ it holds that
    \begin{equation*}
        \begin{aligned}
            \mathbb{E}\left[
            h(\vx_{t+1})
            - h(\vz_t)
            - \mu h(\vx_t)
            + \mu h^*
            \right]
            &\leq
            \frac{1}{2\tau}
            \mathbb{E}\left[
            \|\vx_t - \vz_t\|^2
            -
            \|\vx_{t+1} - \vz_t\|^2
            \right]
            + v,
        \end{aligned}
    \end{equation*}
    where $\mu = 2 \tau L (1+ \gamma)$ and $v = \left(1 + \frac{1}{\gamma}\right)\sigma_*^2 \tau$ for any $\gamma > 0$.
\end{lemma}

Unlike \citet[Lemma 4.2]{garrigos2025last}, the left-hand side depends not only on $\vx_t$ and $\vz_t$ but also on $\vx_{t+1}$. \Cref{theorem41_main} below shows that the technique of \citet{zamani2023exactconvergencerateiterate} accommodates this modification.

We now sketch the proof and highlight the key differences from the unconstrained case.

As discussed in \Cref{sec:starting-point}, the main difficulty is the loss of linearity between $\vx_{t+1}$ and $\vx_t$. Instead of the expansion \eqref{eq:sgd-expansion}, we now have
\begin{equation}\label{eq: tru4}
    \begin{aligned}
        \| \vx_{t+1} - \vz\|^2 = \| \vx_t - \vz\|^2 + 2\langle \vx_{t+1} - \vx_t, \vx_t -\vz \rangle + \| \vx_{t+1} - \vx_t\|^2,
    \end{aligned}
\end{equation}
which is expanding the squared norm. The second term on the right-hand side requires a bound.

The standard analysis of proximal SGD relies on the proximal inequality, which follows from the definition of the proximal operator and convexity of $g$ (see \Cref{proxprojectionlemma}):
\begin{equation}\label{eq: psgd_prox_ineq}
    \begin{aligned}
        \langle \vx_{t+1} - \vx_t + \tau \nabla f_{i_t}(\vx_t), \vz - \vx_{t+1} \rangle \geq \tau(g(\vx_{t+1})- g(\vz)).
    \end{aligned}
\end{equation}
Applying this inequality to the second term in \eqref{eq: tru4} after adding and subtracting $\langle \vx_{t+1} - \vx_t, \vx_{t+1} \rangle$ yields
\begin{equation}\label{eq: psgd_comb}
    \begin{aligned}
        \|\vx_{t+1} - \vz\|^2 - \|\vx_t - \vz\|^2 
        \leq - \|\vx_{t+1} - \vx_t\|^2 + 2\tau (g(\vz) - g(\vx_{t + 1})) 
        + 2\tau \langle \nabla f_{i_t}(\vx_t), \vz - \vx_{t + 1}\rangle.
    \end{aligned}
\end{equation}
The error term $\langle \nabla f_{i_t}(\vx_t), \vz-\vx_{t+1} \rangle$ resembles the corresponding term in \eqref{eq:sgd-expansion}, but with $\vx_t$ in the gradient and $\vx_{t+1}$ in the inner product. This mismatch means we cannot directly apply  convexity. Additional estimation via smoothness is required; see \cref{subsec: one-it-spgd}. Setting $\vz=\vz_t$, as defined below, completes the proof.

\subsubsection{Last-iterate reduction}\label{subsubsec:last-iterate-reduction}
The remainder of the analysis utilizes ideas from \citet{zamani2023exactconvergencerateiterate}, by also using \citet{garrigos2025last}. The next lemma is stated abstractly because it also applies to the incremental proximal method in \cref{sec: rand_inc_prox}.

Following these works, we define the auxiliary sequence $\vz_t=(1-p_t)\vx_t+p_t \vz_{t-1}$ for $t\geq 0$ with $p_0=1$, $\vz_{-1} = \vx^\ast$, and, for $t\ge1$,
\[
p_t=\frac{\mu+T-t+1}{T-t+2}.
\]
The key insight of \citet{zamani2023exactconvergencerateiterate} is that this convex combination, with $\vz_{-1} = \vx^*$, leads us to last-iterate guarantees. Note that $p_t \in [0,1]$ because $\mu < 1$ (where $\mu$ is defined in \Cref{proxobjectivecomb_main}), which holds when $\gamma = \frac{1 - 2 \tau L}{1 + 2 \tau L}$ and $\tau < \frac{1}{2L}$ which is enforced in our step size rule. We now apply this idea to our setting; the proof appears in \Cref{subsec: last_it_spgd}.
\begin{lemma}[Last iterate reduction] \label{theorem41_main}
    Let $h$ be convex. Suppose that for all $t \in [0,T]$, the inequality
    \begin{equation}
        \begin{aligned}
            \mathbb{E}[h(\vx_{t+1}) - h(\vz_t) - \mu h(\vx_t) + \mu h^*]
            \leq \frac{1}{2\tau}\mathbb{E}[\|\vx_t-\vz_t\|^2 - \|\vx_{t+1} - \vz_t\|^2] + \mathcal{E}
        \end{aligned}
    \end{equation}
    holds for some constant $\mathcal{E}$. Then
    \begin{align}
        \alpha_T\E\left[h(\vx_{T+1}) - h^* \right] \leq \frac{1}{2 \tau} \|\vx_0 - \vx^*\|^2 + h(\vx_0) - h^*
        + \mathcal{E}\left(\sum_{t=1}^{T} \alpha_t+\alpha_0 \right),\label{bigTheorem_main}
    \end{align}
    where the sequence $(\alpha_t)$ is defined by $\alpha_{-1} = \alpha_0 = 1$ and
    \[
    \alpha_t
    =
    \frac{T - t + 2}{\mu + T - t + 1} \cdot \alpha_{t-1},
    \qquad t = 1,\dots,T.
    \]
\end{lemma}
From the definition of $\alpha_t$, we have $\alpha_t p_t = \alpha_{t-1}$ for all $t \ge 0$, so $\alpha_t$ is nondecreasing, since $p_t\in[0, 1]$. The convergence rate depends on the growth of $\alpha_t$, which is analyzed in App. \ref{subsec: alpha-bds} using ideas from \citet{garrigos2025last,zamani2023exactconvergencerateiterate}.

Unlike \citet{garrigos2025last}, our recursion involves both $h(\vx_{t+1})$ and $h(\vx_t)$ due to the existence of the proximal operator in the method. We show that the technique of \citet{zamani2023exactconvergencerateiterate} extend naturally to this setting.

Combine \Cref{proxobjectivecomb_main,theorem41_main} with the bounds on $\alpha_t$ from \Cref{subsec: alpha-bds} to obtain \cref{BIGTHEOREM_main}.

\subsection{Corollary for Projected SGD}

\Cref{BIGTHEOREM_main} specializes to projected SGD over a closed convex set $C\subset\R^n$ by setting $g = \delta_C$, where the indicator function $\delta_C(\vx) = 0$ if $\vx \in C$ and $\delta_C(\vx) = +\infty$ otherwise. Since $\delta_C$ is proper, convex, and lower semicontinuous, \cref{assumptions} holds. The problem becomes
\begin{equation*}
    \min_{\vx \in C} f(\vx),
\end{equation*}
and the iteration \eqref{eq:intro-spgd} reduces to
\begin{equation}\label{eq:proj_sgd}
    \vx_{t+1} = \proj_C(\vx_t - \tau \nabla f_{i_t}(\vx_t)), \tag{projSGD}
\end{equation}
where $\proj_C(\vx) = \argmin_{\vy \in C} \|\vy - \vx\|$ is the Euclidean projection onto $C$.

\begin{corollary}[Projected SGD]
    Let \Cref{assumptions} hold for $f$, and let $C$ be convex and closed. For \eqref{eq:proj_sgd}, let $\tau = 1/(3 L \sqrt{T})$. Then,
    \begin{align*}
        \E\left[f(\vx_{T+1}) - f^* \right] =O\left( \frac{\ln (T+1)}{\sqrt{T}} \right).
    \end{align*}
\end{corollary}
We emphasize that the results of the form above for the last iterate under \Cref{assumptions} were only known for SGD for unconstrained problems prior to our work. Since the statement is a special case of \Cref{BIGTHEOREM_main}, its proof is omitted.

\section{Randomized Incremental Proximal Method}\label{sec: rand_inc_prox}
This section extends our analysis to additive regularizers of the form $g = \sum_{i=1}^m g_i$, where we access only the proximal operators of the component functions $g_i$. The standard algorithm for this setting is the randomized incremental proximal method~\eqref{eq:intro-incrementalpgd}, introduced by \citet{bertsekas2011incremental}.

Although this setting generalizes~\eqref{eq:intro-spgd}, the analysis requires an additional assumption: each $g_i$ must be Lipschitz continuous (\Cref{assumptions2_main}). This assumption is standard when proximal operators are accessed stochastically~\citep{bertsekas2011incremental}.

We state the main result below, followed by key proof ideas and a corollary. We then examine why this Lipschitzness assumption is necessary compared to the setting of \Cref{sec: proxsgd}. Complete proofs appear in \Cref{app: ripm}.

\subsection{Statement of the result}

The following theorem establishes a last-iterate convergence rate for \eqref{eq:intro-incrementalpgd}. The proof appears in \cref{subsec: incr_mainthm_proof}.
\begin{theorem}[Last-iterate convergence, incremental proximal]\label{BIGTHEOREM2_main}
    Let \cref{assumptions,assumptions2_main} hold. For \eqref{eq:intro-incrementalpgd}, set $\tau = 1/(5L\sqrt{T})$. Then
    \begin{align*}
        \E[h(\vx_{T+1}) - h^*] \leq \frac{10}{\sqrt{T}} \Big[L D_*^2 + \frac{1}{\sqrt{T}}(h(\vx_0) - h^*) 
        + \frac{\sigma_*^2 + 4m^2L_g^2}{L}\Big(\frac{1}{T}+4\ln(T+1) \Big) \Big].
    \end{align*}
\end{theorem}

This result extends \cref{BIGTHEOREM_main} to handle stochastic access to the regularizer $g$, and requires only the individual proximal operators $\prox_{g_j}$ rather than $\prox_g$. The generalization incurs an additional variance term $4m^2L_g^2$ in the bound and the Lipschitz requirement on each $g_j$; see \cref{assumptions2_main}. As in \cref{BIGTHEOREM_main}, the smooth component $f$ does not require a finite-sum structure.

The most closely related result is \citet{bertsekas2011incremental}, who assumed uniformly bounded gradients $\|\nabla f_i(\vx)\|$ and proved only asymptotic convergence without explicit rates.

\subsection{Proof setup}\label{sec: main_ideas_incr}

We sketch the proof and justify the Lipschitzness assumption on $g_j$, following the framework of \cref{subsec: main_ideas_psgd}.

Instead of \eqref{eq: psgd_prox_ineq}, we now have
\begin{equation}\label{eq: ipgd_prox_ineq}
    \begin{aligned}
        \langle \vx_{t+1} - \vx_t + \tau \nabla f_{i_t}(\vx_t), \vz - \vx_{t+1} \rangle
        \geq \tau m(g_{j_t}(\vx_{t+1})- g_{j_t}(\vz)).
    \end{aligned}
\end{equation}
The aim is to combine this inequality with \eqref{eq: tru4} to obtain a recursion analogous to \eqref{eq: psgd_comb}. The main difference from \cref{subsec: main_ideas_psgd} is that we have $m(g_{j_t}(\vx_{t+1}) - g_{j_t}(\vz))$ instead of $g(\vx_{t+1}) - g(\vz)$. The difficulty becomes apparent:
\begin{align*}
    \mathbb{E}_t[m(g_{j_t}(\vx_{t+1}) - g_{j_t}(\vz))] \neq g(\vx_{t+1}) - g(\vz).
\end{align*}
The iterate $\vx_{t+1}$ depends on $j_t$, so $g_{j_t}$ and $\vx_{t+1}$ are coupled and the conditional expectation does not factor; see~\eqref{eq:intro-incrementalpgd}. The resolution is to decouple them using Lipschitzness:
\begin{align*}
    mg_{j_t}(\vx_{t+1}) &= m g_{j_t}(\vx_t) + m(g_{j_t}(\vx_{t+1}) - g_{j_t}(\vx_{t})) \\
    &\leq m g_{j_t}(\vx_t) + mL_g\|\vx_{t+1} - \vx_t\|.
\end{align*}
With $g_{j_t}$ evaluated at $\vx_t$, which is $\mathcal{F}_t$-measurable, we have $\mathbb{E}_t[mg_{j_t}(\vx_t)] = g(\vx_t)$. A second application of Lipschitzness then recovers $g(\vx_{t+1})$. This is the standard argument for handling such coupling \citep{bertsekas2011incremental}.

This decoupling introduces additional error terms on the right-hand side of the one-iteration bound, compared to \cref{proxobjectivecomb_main}. See App. \ref{subsec: one-it-proof-incr} for the proof of the next lemma.
\begin{lemma}[Per-iteration descent, incremental proximal] \label{blockproxlinearcomb_main}
    Let \cref{assumptions,assumptions2_main} hold. For \eqref{eq:intro-incrementalpgd}, set $\tau < 1/(4L)$.
    Then for all $t = 0, \ldots, T$ and $\vz_t \in \mathcal{F}(\vx_0, \ldots, \vx_t)$ it holds that
    \begin{equation*}
        \begin{aligned}
            \mathbb{E}_t\left[
            h(\vx_{t+1})
            - h(\vz_t)
            - \mu h(\vx_t)
            + \mu h^*
            \right]
            \leq
            \frac{1}{2\tau}
            \mathbb{E}_t\left[
            \|\vx_t - \vz_t\|^2
            -
            \|\vx_{t+1} - \vz_t\|^2
            \right]
            + v + 8 \tau m^2 L_g^2,
        \end{aligned}
    \end{equation*}
    where $\mu = 2 \tau L (1+ \gamma)$ and $v = \left(1 + \frac{1}{\gamma}\right)\sigma_*^2 \tau$ for any $\gamma > 0$.
\end{lemma}
Applying \cref{theorem41_main} with $\mathcal{E} = v + 8\tau m^2 L_g^2$ completes the proof of \cref{BIGTHEOREM2_main}.

\subsection{Corollary for Stochastic Proximal Point Method}\label{sec: spp}

The stochastic proximal point (SPP) method applies to \eqref{eq:problem} when $f\equiv 0$. Given an objective $g(\vx) := \frac{1}{m}\sum_{j=1}^m g_j(\vx)$, the method generates iterates
\begin{align*}
    \vx_{t+1} = \prox_{\tau g_{j_t}}(\vx_t),
\end{align*}
where $j_t$ is sampled uniformly at random from $\{1, \ldots, m\}$.

Though well-studied, we are not aware of an optimal last-iterate rate guarantee for this method under mere convexity. The proof of the next result is given in \Cref{app: spp_proof}.
\begin{corollary}[Stochastic proximal point]\label{cor: spp}
    Let \Cref{assumptions2_main} hold. For the stochastic proximal point method, set $\tau = 1/(\sqrt{T+1})$.
    Then, we have
    \begin{align*}
        \E[g(\vx_{T+1}) - g(\vx^*)] =O\left( \frac{\ln(T+1)}{\sqrt{T}} \right).
    \end{align*}
\end{corollary}
A suboptimal $O(T^{-1/3})$ rate was shown for this method in \citep{toulis2021proximal}.
The other closest prior result is \citet[Section~3.2]{cai2025last}, which analyzes \emph{cyclic} incremental proximal point for finite-sum objectives $\frac{1}{N} \sum_{i=1}^N g_i(\vx)$. To enable direct comparison, we apply our analysis to the same normalized formulation. Our rate improves upon theirs by a factor of $\sqrt{N}$. We analyze SPP with i.i.d.\@ sampling rather than cyclic selection. Their bound is deterministic whereas ours holds in expectation, making the results complementary.

\begin{figure*}[t]
    \centering

    \begin{subfigure}[t]{0.3\linewidth}
        \centering
        \includegraphics[width=\linewidth]{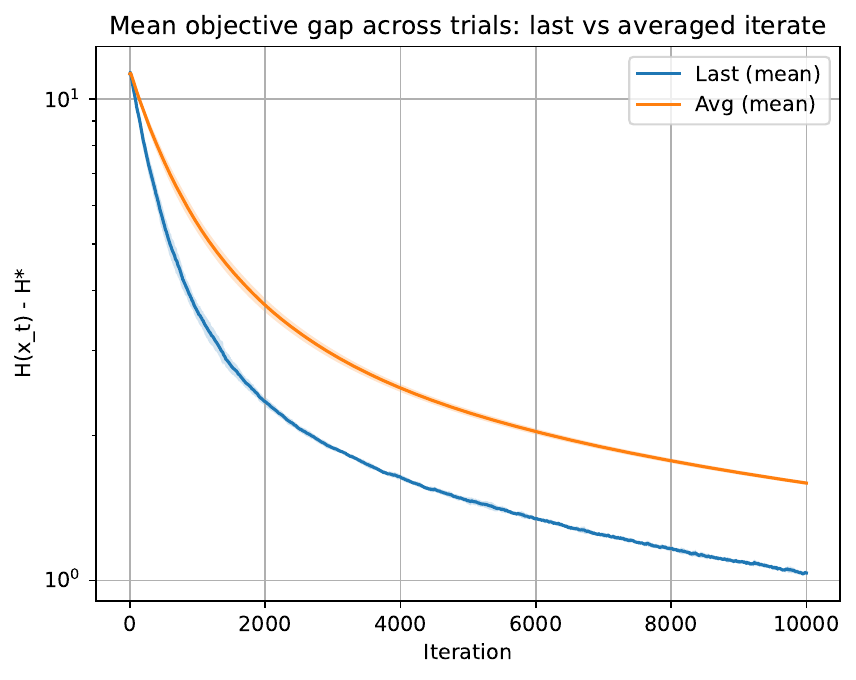}
        \caption{Synthetic Lasso objective gap $H(\vx_t)-H^*$ (semilog scale).}
        \label{fig:lasso_mean_gap}
    \end{subfigure}
        \begin{subfigure}[t]{0.3\linewidth}
        \centering
        \includegraphics[width=\linewidth]{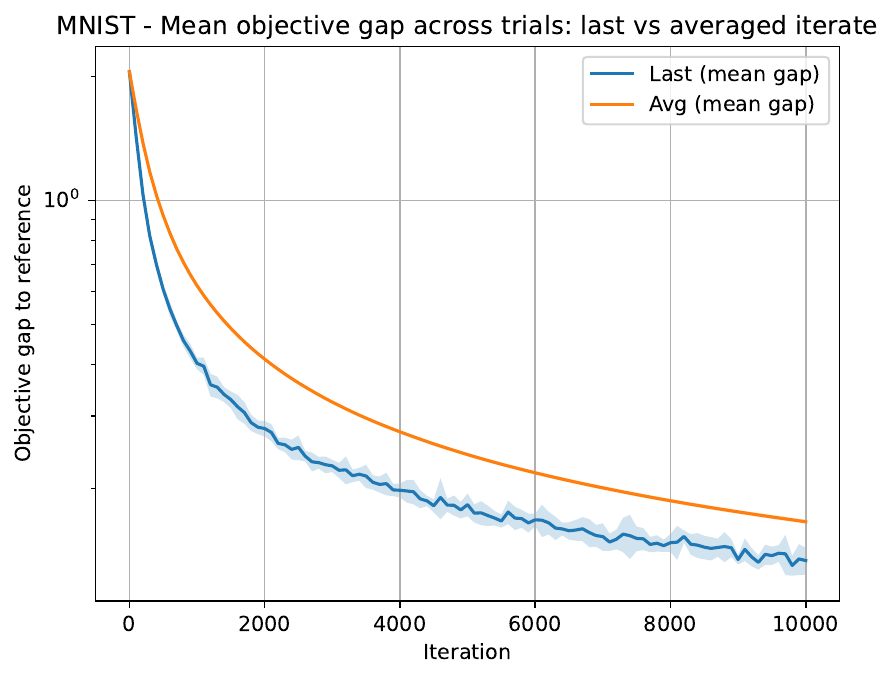}
        \caption{MNIST logistic regression: objective gap to reference solution (semilog scale).}
        \label{fig:mnist_mean_gap}
    \end{subfigure}
    
    \caption{Comparison of last and averaged iterates.}
    \label{fig:lasso_two_plots}
\end{figure*}

\subsection{Extensions to BlockProx}\label{sec: blockprox}

We now apply the incremental proximal framework to BlockProx, a distributed algorithm for graph-structured regularizers. Consider the optimization problem with partially separable objectives
\begin{equation*}
\min_{\vx}\ h(\vx) := \textstyle\sum_{i=1}^N f_i(\vx_i) + \sum_{j=1}^m g_j(\vx),
\end{equation*}
where each $\vx_i\in\mathbb{R}^d$ denotes the variables at node $i$ and $\vx=(\vx_1, \dots, \vx_N)\in\mathbb{R}^{Nd}$ joins all node variables. This matches our framework with $f(\vx) = \sum_{i=1}^N f_i(\vx_i)$ and $g(\vx) = \sum_{j=1}^m g_j(\vx)$, where each $f_i\colon\mathbb{R}^d \to \mathbb{R}$ is convex and smooth, and each $g_j$ is convex and $L_g$-Lipschitz. For each regularizer $g_j$, let
\[
    S_j \subseteq \{1,\ldots,N\}
\]
denote the support of $g_j$, namely the set of nodes that $g_j$ depends on. That is, $g_j(\vx) = g_j(\vx_{S_j})$ where $\vx_{S_j} = (\vx_i \mid i \in S_j)$.

In particular, $m$ may equal the number of edges in a graph, with $g_j$ coordinating communication between nodes sharing edge $j$.

The gradient $\nabla f(\vx) = (\nabla f_1(\vx^{(1)}), \ldots, \nabla f_N(\vx^{(N)}))$ decomposes by separability. The BlockProx algorithm of \citet{LinKuangAlacaogluFriedlander2025} iterates as
\begin{equation}
\begin{aligned}
\vy_t &= \vx_t - \tau \nabla f(\vx_t), \\
\vx_{t+1}^{(i)} &= \left[ \prox_{m\tau g_{j_{i, t}}}(\vy_t) \right]_i \in \mathbb{R}^d, \text{~if~} i\in S_{j_{i, t}},\\
\vx_{t+1}^{(i)} &= \vy_t^{(i)},\text{~if~} i\not\in S_{j_{i, t}},
\end{aligned}
\end{equation}
where $j_{i, t}$ is drawn uniformly at random from $1,\dots,m$, which is the edge that node $i$ selects at iteration $t$. All the nodes are allowed to sample another edge and this process runs over all the nodes in parallel.

The analysis in \citet{LinKuangAlacaogluFriedlander2025} reduces BlockProx to \ref{eq:intro-incrementalpgd} without stochastic gradients. Specifically, Proposition~5.2 of that work translates convergence results for \ref{eq:intro-incrementalpgd} to BlockProx.

In particular, that work showed BlockProx achieves $\widetilde{O}(1/\sqrt{T})$ convergence either \emph{(i)} on the last iterate when the gradients of $f$ are bounded \citep[Thm.~5.6]{LinKuangAlacaogluFriedlander2025}, or \emph{(ii)} on the average iterate when $f$ is smooth \citep[Thm.~5.9]{LinKuangAlacaogluFriedlander2025}. However, the experiments therein use linear least squares for $f$ and output the last iterate. Hence, the existing analysis does not cover the method as implemented. We bridge this theory-practice gap by proving last-iterate rates when each $f_i$ is convex and smooth, without bounded gradients.
The proof of the following theorem appears in \cref{app: blockprox}.
\begin{theorem}[Last-iterate convergence, BlockProx]\label{th: blockprox}
Let \cref{assumptions,assumptions2_main} hold. Consider the iterates of BlockProx with $\tau = 1/(3L\sqrt{T})$. Then, we have
\begin{equation*}
\mathbb{E}[h(\vx_{T}) - h^*] = O\left(\frac{\ln(T+1)}{\sqrt{T}} \right).
\end{equation*}
\end{theorem}
Although \ref{eq:intro-incrementalpgd} generalizes BlockProx by including stochastic gradients, BlockProx uses deterministic gradients of $f$. This simpler structure admits a direct proof: we adapt the one-iteration analysis in \citet{LinKuangAlacaogluFriedlander2025} to the framework of \citet{garrigos2025last}.

\section{Numerical Experiments}\label{sec: experiments}
\begin{figure*}[t]
    \centering
    \begin{subfigure}[t]{0.33\linewidth}
        \centering
        \includegraphics[width=\linewidth]{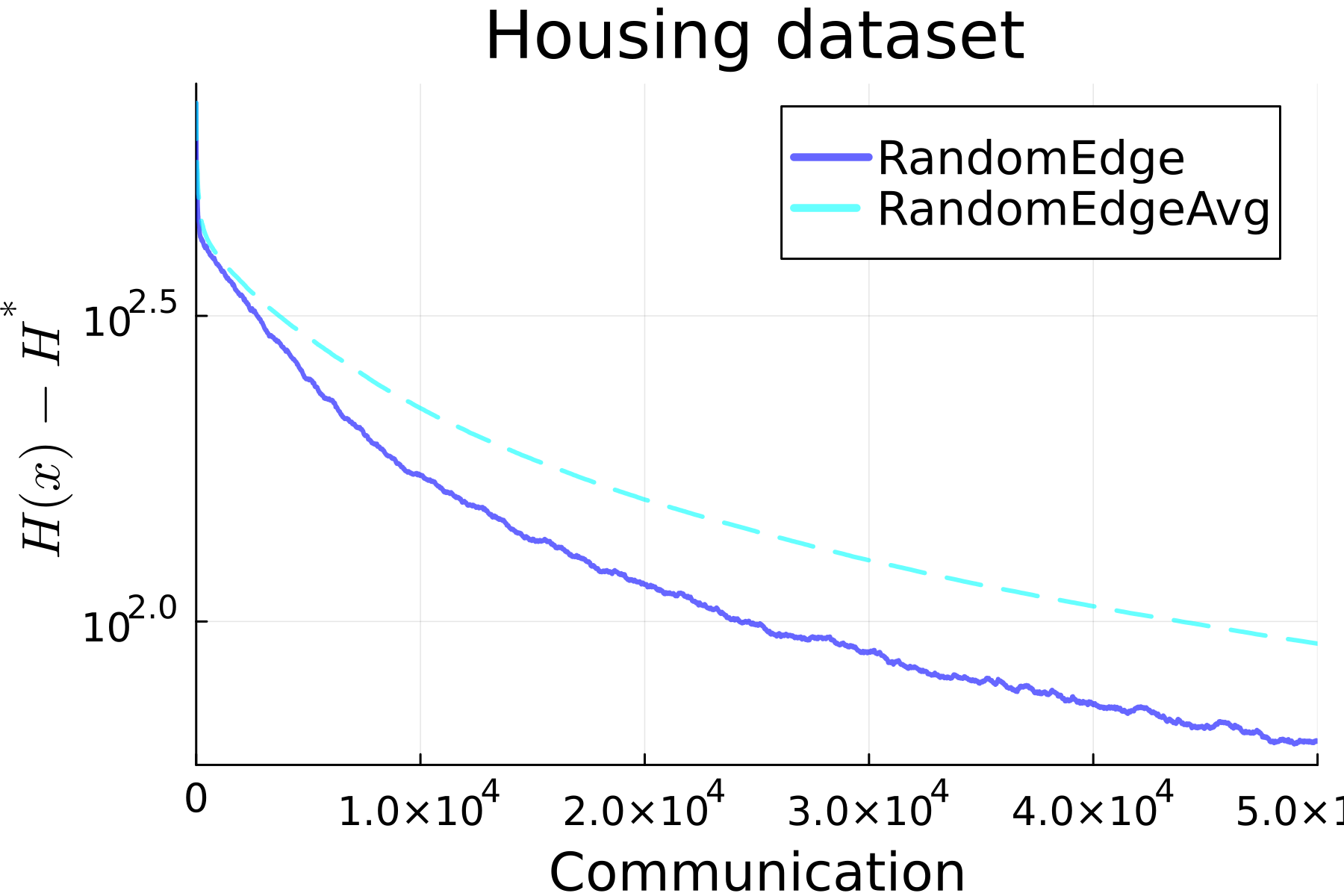}
        \caption{Sacramento housing dataset.}
        \label{fig:BlockProx_housing}
    \end{subfigure}
    \begin{subfigure}[t]{0.33\linewidth}
        \centering
        \includegraphics[width=\linewidth]{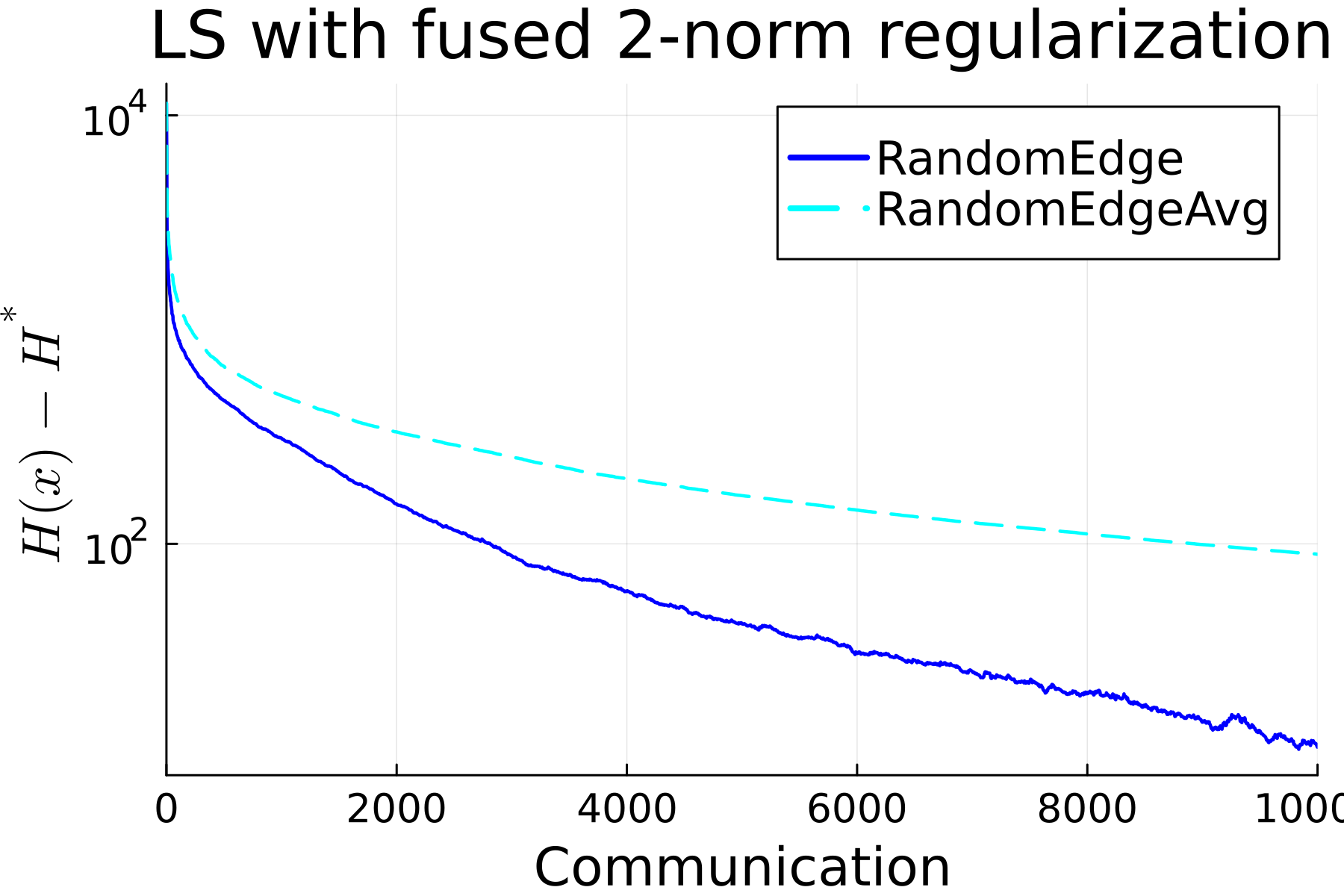}
        \caption{Synthetic lasso with 5 groups, $p=2$.}
        \label{fig:BlockProx_LS_2_5_75}
    \end{subfigure}
    \begin{subfigure}[t]{0.33\linewidth}
        \centering
        \includegraphics[width=\linewidth]{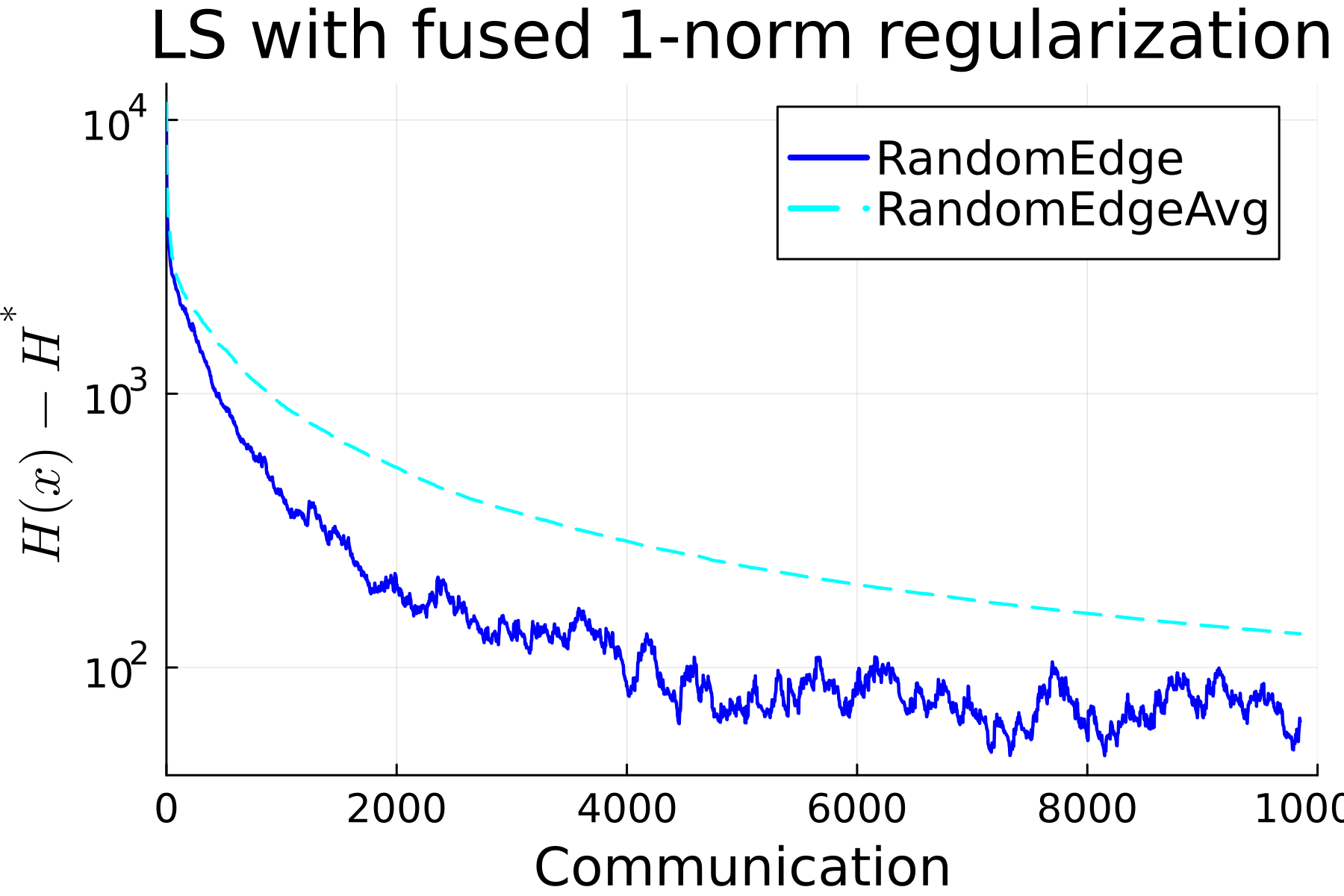}
        \caption{Synthetic lasso with 5 groups, $p=1$.}
        \label{fig:BlockProx_LS_1_5_75}
    \end{subfigure}
    \caption{Comparison of last and averaged iterates over communication rounds for BlockProx.}
    \label{fig:BlockProx_comparison}
\end{figure*}

Practitioners routinely output the last iterate of SGD-based methods rather than
the average. 
We provide experiments that complement our theoretical
analysis to show that this common practice is indeed useful.

We compare the last iterate $\vx_T$ and the averaged iterate
$\bar{\vx}_T = (1/T)\sum_{t=1}^T \vx_t$ on two problems: Lasso regression
with synthetic data and $\ell_1$-regularized logistic regression on MNIST.
All results are averaged over 10 independent trials with fixed datasets.

\paragraph{Synthetic Lasso.}
We solve
\begin{align*}
    \min_{\vx} \left\{h(\vx)=\frac{1}{2n}\|A\vx-\vy\|_2^2+\lambda\|\vx\|_1\right\},
\end{align*}
using a synthetic matrix $A$ and sparse ground truth. We compute the optimal
value $h^*$ using \texttt{CVXPY} and plot the optimality gap $h(\vx_t) - h^*$
over iterations. The stepsize is $\tau = 1/(4L\sqrt{T})$, consistent with
our theoretical prescription.
\Cref{fig:lasso_mean_gap} shows that while the averaged iterate decreases
more smoothly, the last iterate achieves a smaller optimality gap across
all trials.

\paragraph{$\ell_1$-Regularized Logistic Regression.}
We solve multiclass logistic regression with $\ell_1$ regularization on MNIST
using~\eqref{eq:intro-spgd} with stepsize $\tau = \Theta(1/\sqrt{T})$. That is, 
\[
    \min_{\mX \in \R^{d \times K}}\left\{h(\mX) = \frac{1}{n}\sum_{i=1}^n\ell(\mX; a_i,y_i) + \lambda \|\mX\|_1 \right\},
\]
where 
\[\ell(\mX;a_i,y_i) =-\langle \mX_{y_i},a_i\rangle + \log\sum_{k=1}^K e^{\langle \mX_k,a_i\rangle}\] 
is the multiclass logistic loss, $a_i \in \R^d$ is the $i$-th data point, $y_i \in \{1, \ldots, K\}$ is its label and $\mX_{y_i}$ is the weight vector corresponding to class $y_i$. For MNIST, $d = 784$ and $K = 10$. 
We compute a reference solution $h^*$ using the SAGA solver from Scikit-Learn \citep{Pedregosa_Scikit-learn_Machine_Learning_2011}.
\Cref{fig:mnist_mean_gap} confirms that the last iterate typically achieves
a smaller optimality gap than the averaged iterate, consistent with the
Lasso experiment.

\paragraph{Graph-structured regularization (BlockProx).}
We solve a graph-structured optimization problem using the BlockProx algorithm of \citet{LinKuangAlacaogluFriedlander2025}. Given a graph $\mathcal{G} = (\mathcal{V}, \mathcal{E})$ with $|\mathcal{V}| = n$ and the edges $\mathcal{E}$ capture the communication structure, we solve
\begin{align*}
    \min_{\vx} \left\{h(\vx) = \sum_{i \in \mathcal{V}} f_i(\vx_i) + \sum_{(i,j) \in \mathcal{E}} g_{ij}(\vx_i, \vx_j)\right\},
\end{align*}
In particular, we focus on the network lasso setting \citep{hallac2015network}, where
\[
    f_i(\vx_i) = \frac{1}{2}\|A_i \vx_i - b_i\|^2,
    \quad
    g_{ij}(\vx_i,\vx_j) = \lambda \|\vx_i - \vx_j\|_p,
\]
where $p \in \{1,2\}$. For each node $i$, the data matrix $A_i$ is generated either synthetically or from the Sacramento housing dataset (only for $p=2$) \citep{hallac2015network}. 
For the synthetic data, we consider three communication network settings:
\begin{enumerate}
    \item five groups with $10$, $17$, $18$, $18$, and $12$ nodes, respectively; (see \cref{fig:BlockProx_comparison} (b) and (c))
    \item one group with $20$ nodes; (see \cref{fig:BlockProx_LS_Addn} in \cref{app: additionalexp})
    \item one group with $40$ nodes, where the communication network is fully connected,
    i.e., every pair of nodes is connected by an edge. (see \cref{fig:BlockProx_LS_Addn} in \cref{app: additionalexp})
\end{enumerate}
Additional details regarding the generation of the synthetic data matrix and experimental setup are given in Sections 6.1 and 6.2 of \citet{LinKuangAlacaogluFriedlander2025}. We compare the last iterate produced by BlockProx with the averaged iterate. \Cref{fig:BlockProx_comparison} shows that the last iterate achieves a smaller optimality gap than the averaged iterate across all settings, consistent with our previous experiments.

\section{Conclusions and Perspectives}

For \ref{eq:intro-spgd} and \ref{eq:intro-incrementalpgd} (that generalize SPP), we proved that the last-iterate has the rate of convergence $\widetilde{O}( 1/\sqrt{T} )$ with no bounded variance assumptions, under componentwise smoothness and convexity. This rate is $\ln (T+1)$ away from the optimal rate that is  known with bounded variance, and it matches the best-known rates for last-iterate of SGD applied to unconstrained minimization \citep{garrigos2025last, attia2025fastlastiterateconvergencesgd} or for averaged iterate for problems with unbounded variance \citep{khaled2023unified}.

Our work paves the way to many different directions. The first direction is extending our guarantees to methods with random reshuffling or sampling without replacement for selecting stochastic gradients. For this, the techniques developed in \citet{cai2025last} and \citet{liu2025improved} will likely prove useful. The second direction is lifting the Lipschitzness requirement in \cref{sec: rand_inc_prox}. Indeed this assumption is common in the SPP literature. However, for the important special case of $g_i(\vx) = \delta_{C_i}(\vx)$ for a convex and closed set $C_i$, SPP reduces to the well-known method of randomized alternating projections, which admits a separate analysis framework, see for example \cite{nedic2011random}. Extending these techniques to our template would allow us to solve problems of the form $\min\ \{ f(\vx)  \mid \vx\in \bigcap_{i=1}^m C_i\}$ with last-iterate convergence rates and without bounded variance assumptions on the smooth $f$. Two other directions are extending our analysis to allow Bregman distances or to show high probability guarantees on the last iterate.
\newpage
\section*{Acknowledgements}
Michael P. Friedlander acknowledges the support of the Natural Sciences and Engineering Research Council
of Canada (NSERC).

Ahmet Alacaoglu acknowledges the support of the Natural Sciences and Engineering Research Council of
Canada (NSERC), [funding reference number RGPIN-2025-06634].

\section*{Impact Statement}

``This paper presents work whose goal is to advance the field of Machine
Learning. There are many potential societal consequences of our work, none
which we feel must be specifically highlighted here.''

\bibliography{last_it_prox_arxiv.bib}
\bibliographystyle{abbrvnat}

\newpage
\appendix
\crefalias{section}{appendix}
\crefalias{subsection}{appsubsection}
\onecolumn

\section{Proofs for Section \ref{sec: proxsgd}}\label{app: spgd}
We will now present our main last-iterate results for \eqref{eq:intro-spgd}. For the statement in the main text, that is, \Cref{BIGTHEOREM_main}, we used $C=3$ and simplified the bounds.
\begin{theorem}[Last-iterate convergence, proximal SGD, generalized form of \Cref{BIGTHEOREM_main}]\label{BIGTHEOREM}
    Consider the \eqref{eq:intro-spgd}  algorithm with step size $\tau = \frac{1}{C L \sqrt{T}}$ for some $C > 2$. Suppose that assumption \ref{assumptions} hold. Then, the last iterate $\vx_{T+1}$ satisfies
    \begin{align*}
        \E\left[h(\vx_{T+1}) - h^* \right] &\leq \exp\left(\frac{8}{e C}\right) \left[\frac{C L \|\vx_0 - \vx^*\|^2}{\sqrt{T}} + \frac{2}{T} [h(\vx_0) - h^*] + \frac{4 \sigma_*^2 }{(C-2)LT^{1.5}} + \frac{16 \sigma_*^2  \ln (T+1)}{(C-2)L\sqrt{T}} \right].
    \end{align*}
\end{theorem}
We shall now provide the proof of this theorem. Our proof proceeds in four stages:
\begin{enumerate}
    \item We prove a technical variance transfer lemma that allows us to bound $\E[\|\nabla f_i(\vx)\|^2]$. See \Cref{variancetransfer}.
    \item We prove a one-step inequality that compares the iterates to a reference point $\vz_t \in \mathcal{F}(\vx_0, \dots \vx_t)$. See \Cref{onestepid}
    \item We use this to prove a bound for the linear combination of objective values including the evaluation at a reference point $\vz_t$ which is a convex combination of our iterates. See \ref{proxobjectivecomb_main}.
    \item Using the above, we chose the structure of $\vz_t$ and weights $\alpha_t$ to obtain a bound for the last iterate.
\end{enumerate}
We start with the variance transfer lemma, which is now classical in the SGD literature, see for example \citep[Lemma A.2]{khaled2023unified}. 
\begin{lemma}[Variance transfer] \label{variancetransfer}
    Let \Cref{assumptions} hold. Then for every $\gamma > 0$, we have
    \begin{equation}
        \E_t\|\nabla f_{i_t}(\vx_t)\|^2 \leq 2L (1+ \gamma) [h(\vx_t) - h^*] + \left(1 + \frac{1}{\gamma}\right) \sigma_*^2.
    \end{equation}
\end{lemma}
\begin{proof}
    Let us denote by $\vx^*$ an arbitrary solution of our problem. By smoothness and convexity of $f_i$, we have the standard inequality \citep[Thm.~2.1.5]{nesterov2018lectures}, see also \eqref{eq:cocoercivity},
    \begin{align*}
        \frac{1}{2L} \|\nabla f_{i_t}(\vx_t) - \nabla f_{i_t}(\vx^*)\|^2 &\leq f_{i_t}(\vx_t) - f_{i_t}(\vx^*) - \langle \nabla f_{i_t}(\vx^*), \vx_t-\vx^*\rangle.
    \end{align*}
 We take expectation to get
    \begin{align}\label{eq: var_tran_1}
        \frac{1}{2L} \E_t\|\nabla f_{i_t}(\vx_t) - \nabla f_i(\vx^*)\|^2 &\leq f(\vx_t) - f(\vx^*) - \langle \nabla f(\vx^*), \vx_t - \vx^*\rangle.
    \end{align}
    From the optimality condition of our problem \eqref{eq:problem}, we have that $0 \in \nabla f(\vx^*) + \partial g(\vx^*)$. Hence, the convexity of $g$ yields
    \begin{align*}
        g(\vx_t) \geq g(\vx^*) + \langle -\nabla f(\vx^*),  \vx_t - \vx^*\rangle.
    \end{align*}
Substituting into \eqref{eq: var_tran_1} gives
    \begin{equation}\label{eq: var_tran2}
        \begin{aligned}
            \frac{1}{2L} \E_t\|\nabla f_{i_t}(\vx_t) - \nabla f_{i_t}(\vx^*)\|^2 \leq f(\vx_t) - f(\vx^*) + g(\vx_t) - g(\vx^*) = h(\vx_t) - h^*.
        \end{aligned}
    \end{equation}
    Here, we use Young's inequality to estimate as
    \begin{align*}
        \E_t\|\nabla f_{i_t}(\vx_t)\|^2&\leq (1 + \gamma) \E_{t}\|\nabla f_{i_t}(\vx_t) - \nabla f_{i_t}(\vx^*)\|^2 + \left(1 + \frac{1}{\gamma}\right) \E_t\|\nabla f_{i_t}(\vx^*)\|^2.
    \end{align*}
We use \eqref{eq: var_tran2} and \Cref{assumptions} to bound the first and second terms on the right-hand side, respectively. 
\end{proof}

We now restate the optimality condition for the proximal operator that we use in our proofs, which is well known. Its proof follows from the definition of $\prox_{\tau g}$ and the convexity of $g$. See for example \citep[Theorem 6.39]{beck2017first}.
\begin{lemma}[Prox-inequality]\label{proxprojectionlemma}
    For any proper, convex, and lower semicontinuous function $g: \R^n \to (-\infty, +\infty]$, we have
    \begin{equation}
\vu = \prox_{\tau g}(\vy) \iff \langle \vu - \vy, \vx-\vu \rangle \geq \tau (g(\vu) - g(\vx)),~~~\text{for all~} \vx \in \R^n.
    \end{equation}
\end{lemma}
We finally provide a one-step progress inequality with an arbitary reference point $\vz$. 
\begin{lemma}[One-step progress] \label{onestepid}
Let \Cref{assumptions} hold. Then, for the iterates $(\vx_t)$ generated by \eqref{eq:intro-spgd}, for any $\vz \in \R^n$ and $\delta_1>0$, it holds that
        \begin{align*}
        \|\vx_{t+1} - \vz_t\|^2 - \|\vx_t - \vz_t\|^2 &\leq 2\tau \left[f_{i_t}(\vz_t) - f_{i_t}(\vx_t) + g(\vz_t) - g(\vx_{t + 1}) + f(\vx_t) - f(\vx_{t+1})\right] \\
        &\quad+
        \frac{\tau}{ \delta_1} \|\nabla f_{i_t}(\vx_t) - \nabla f(\vx_t)\|^2 + (\tau \delta_1 + \tau L - 1) \|\vx_t - \vx_{t+1}\|^2.
    \end{align*}
\end{lemma}
\begin{proof}
    Using the standard identity, $\langle \va, \vb \rangle = -\frac{1}{2} \| \va\|^2 - \frac{1}{2} \| \vb\|^2+\frac{1}{2} \| \va+\vb\|^2$, we have
    \begin{align}\label{eq:prox1}
        \|\vx_{t+1} - \vz_t\|^2 - \|\vx_t - \vz_t\|^2 &= \|\vx_{t+1} - \vx_t\|^2 + 2\langle \vx_{t+1} - \vx_t, \vx_t - \vz_t\rangle.
    \end{align}
For the last term on the right-hand side, we use Lemma \ref{proxprojectionlemma} with $\vu=\vx_{t+1}$, $\vy= \vx_t - \tau \nabla f_{i_t}(\vx_t)$, and $\vx = \vz_t$ to obtain
    \begin{align*}
        2\langle \vx_{t+1} - \vx_t, \vx_t - \vz_t \rangle &= 2\langle \vx_{t+1} - \vx_t, \vx_t - \vx_{t+1}\rangle + 2\langle \vx_{t+1} - \vx_t, \vx_{t+1} - \vz_t\rangle \\
        &\leq - 2\|\vx_{t+1} - \vx_{t}\|^2 + 2\tau \left(g(\vz_t) - g(\vx_{t + 1}) +  \langle \nabla f_{i_t}(\vx_t), \vz_t - \vx_{t + 1}\rangle\right).
    \end{align*}
    Substituting this into \eqref{eq:prox1} gives
    \begin{align}
        \|\vx_{t+1} - \vz_t\|^2 - \|\vx_t - \vz_t\|^2 &\leq - \|\vx_{t+1} - \vx_t\|^2 + 2\tau (g(\vz_t) - g(\vx_{t + 1})) +2\tau \langle \nabla f_{i_t}(\vx_t), \vz_t - \vx_{t + 1}\rangle. \label{eq: spgd_term2}
    \end{align}
    We use convexity of $f_{i_t}$ to bound the last term on the right-hand side:
    \begin{align}
        2\tau \langle \nabla f_{i_t}(\vx_t), \vz_t - \vx_{t + 1}\rangle &= 2\tau \langle \nabla f_{i_t}(\vx_t), \vz_t - \vx_{t}\rangle+2\tau \langle \nabla f_{i_t}(\vx_t), \vx_{t} - \vx_{t+1}\rangle \notag \\
        &\leq 2\tau [f_{i_t}(\vz_t) - f_{i_t}(\vx_t)] + 2\tau \langle \nabla f_{i_t}(\vx_t), \vx_t - \vx_{t + 1}\rangle.\label{eq: spgd_344}
    \end{align}
    Adding and subtracting $\nabla f(\vx_t)$ in the inner product gives us:
    \begin{align}\label{eq: spgd_34}
        2\tau \langle \nabla f_{i_t}(\vx_t), \vx_t - \vx_{t + 1}\rangle=2\tau \langle \nabla f_{i_t}(\vx_t) - \nabla f(\vx_t), \vx_t - \vx_{t + 1}\rangle + 2\tau \langle \nabla f(\vx_t), \vx_t - \vx_{t + 1}\rangle.
    \end{align}
    We estimate the first term on the right-hand side by using the Young's inequality:
    \begin{equation*}
        2\tau \langle \nabla f_{i_t}(\vx_t) - \nabla f(\vx_t), \vx_t - \vx_{t+1} \rangle\leq \frac{\tau}{\delta_1} \|\nabla f_{i_t}(\vx_t) - \nabla f(\vx_t)\|^2 + \tau \delta_1 \| \vx_t-\vx_{t+1}\|^2.
    \end{equation*}
    For the second term on the right-hand side of \eqref{eq: spgd_34}, we can use the smoothness of $f$ to derive
    \begin{equation*}
        2\tau \langle \nabla f(\vx_t), \vx_t - \vx_{t + 1}\rangle \leq 2\tau \Big( f(\vx_t) - f(\vx_{t+1}) + \frac{L}{2} \|\vx_t - \vx_{t+1}\|^2\Big).
    \end{equation*}
    We plug in the last two estimates to \eqref{eq: spgd_34} which we then insert in \eqref{eq: spgd_344} to get the estimate
    \begin{align*}
2\tau \langle \nabla f_{i_t}(\vx_t), \vz_t - \vx_{t + 1}\rangle &\leq  2\tau [f_{i_t}(\vz_t) - f_{i_t}(\vx_t)] + 
        \frac{\tau}{ \delta_1} \|\nabla f_{i_t}(\vx_t) - \nabla f(\vx_t)\|^2  \\
        &\quad+\tau \delta_1\|\vx_t - \vx_{t+1}\|^2 + 2\tau \Big(f(\vx_t) - f(\vx_{t+1}) + \frac{L}{2} \|\vx_t - \vx_{t+1}\|^2\Big).
    \end{align*}
 Using this bound in \eqref{eq: spgd_term2} and rouping the terms involving $\|\vx_t - \vx_{t+1}\|^2$ gives the assertion.
\end{proof}
\subsection{Bounding a linear combination of function values}\label{subsec: one-it-spgd}
We now prove \Cref{proxobjectivecomb_main}. We restate it here for the reader's convenience.
\begin{lemma}[Per-iteration progress, see \Cref{proxobjectivecomb_main}]
    Let \Cref{assumptions} hold. In \eqref{eq:intro-spgd}, pick $\tau$ such that $\tau < 1/(2L)$.
    Then for all $t\in[0, T]$ and $\vz_t$ in $\mathcal{F}(\vx_0, ..., \vx_t)$ it holds that
    \begin{equation*}
            \mathbb{E}\left[
            h(\vx_{t+1})
            - h(\vz_t)
            - \mu h(\vx_t)
            + \mu h^*
            \right]\leq
            \frac{1}{2\tau}
            \mathbb{E}\left[
            \|\vx_t - \vz_t\|^2
            -
            \|\vx_{t+1} - \vz_t\|^2
            \right]
            + \tau\left(1 + \frac{1}{\gamma}\right)\sigma_*^2 ,
    \end{equation*}
    where $\mu = 2 \tau L (1+ \gamma)$  for any $\gamma > 0$.
\end{lemma}
\begin{proof}
    On the result of Lemma \ref{onestepid}, 
    we take conditional expectation and rearrange (noting the cancellation of the terms $\mathbb{E}_t[f_{i_t}(\vx_t)]$ and $f(\vx_t)$) to derive
    \begin{align*}
        2\tau \E_t[f(\vx_{t+1}) + g(\vx_{t+1}) - f(\vz_t) - g(\vz_t)] &\leq \|\vx_t - \vz_t\|^2 - \E_t\|\vx_{t+1} - \vz_t\|^2 +  \frac{\tau}{ \delta_1} \E_t\|\nabla f_{i_t}(\vx_t) - \nabla f(\vx_t)\|^2 \\
        &\quad + (\tau \delta_1 + \tau L - 1) \E_t \|\vx_t - \vx_{t+1}\|^2.
    \end{align*}
    We next use the bound $\mathbb{E}_{t}\|\nabla f_{i_t}(\vx_t) - \nabla f(\vx_t)\|^2 \leq \mathbb{E}_{t}\|\nabla f_{i_t}(\vx_t)\|^2$
    and Lemma \ref{variancetransfer} to get
    \begin{align}
        2\tau \E_t[f(\vx_{t+1}) + g(\vx_{t+1}) - f(\vz_t) - g(\vz_t)] &\leq \|\vx_t - \vz_t\|^2 - \E_t\|\vx_{t+1} - \vz_t\|^2 \notag  \\ 
        &\quad +  \frac{\tau}{ \delta_1} \left(2L (1+ \gamma) [h(\vx_t) - h^*] + \left(1 + \frac{1}{\gamma}\right) \sigma_*^2\right) \notag \\
        &\quad + (\tau \delta_1 + \tau L - 1) \E_t \|\vx_t - \vx_{t+1}\|^2.\label{eq: suj4}
    \end{align}
    We now  set $\delta_1 = \frac{1}{\tau} - L > 0$ so that the last term in \eqref{eq: suj4} becomes nonpositive. 
    We next use $\frac{\tau}{\delta_1} = \frac{\tau^2}{1 - \tau L} \leq 2\tau ^ 2$, which is by $\tau L < \frac{1}{2}$, and divide both sides by $2\tau$ to get
    \begin{align*}
        \E_t[f(\vx_{t+1}) + g(\vx_{t+1}) - f(\vz_t) - g(\vz_t)] &\leq \frac{1}{2\tau} \left(\|\vx_t - \vz_t\|^2 - \mathbb{E}_t\|\vx_{t+1} - \vz_t\|^2\right) \\ 
        &\quad +  2\tau L (1+ \gamma) (h(\vx_t) - h^*) + \tau\left(1 + \frac{1}{\gamma}\right) \sigma_*^2.
    \end{align*}
    Rearranging terms, grouping $h(\vx) = f(\vx) + g(\vx)$ and using the notation $\mu=2\tau L(1+\gamma)$, we get
    \begin{align*}
        \E_t[h(\vx_{t+1}) - h(\vz_t) - \mu h(\vx_t) + \mu h^*] &\leq \frac{1}{2\tau} \left(\|\vx_t - \vz_t\|^2 - \mathbb{E}_t\|\vx_{t+1} - \vz_t\|^2\right) +\tau\left(1 + \frac{1}{\gamma}\right) \sigma_*^2.
    \end{align*}
    Taking total expectation gives us the required bound. 
\end{proof}

\subsection{Reduction to last iterate bounds}\label{subsec: last_it_spgd}
The following result provides a last iterate bound for $\E [h(\vx_T) - h^*]$. It follows the same arguments as in \cite{garrigos2025last} and \cite{zamani2023exactconvergencerateiterate}, but with different constants and terms due to the existence of a different left-hand side including $\vx_{t+1}, \vx_t, \vz_t$.
\begin{lemma}[Last-iterate reduction, see \Cref{theorem41_main}] \label{theorem41}
    Suppose that for $t = 0,\dots,T$ and every $\vz_t \in \mathcal{F}(\vx_0,\ldots,\vx_t)$ it holds that
    \begin{equation}\label{proxcomb111}
    \E[h(\vx_{t+1}) -  h(\vz_t) - \mu h(\vx_t) + \mu h^*] \leq \frac{1}{2\tau} \E[\|\vx_t - \vz_t\|^2 - \|\vx_{t+1} - \vz_t\|^2] + v,
    \end{equation}
    for some $0 \leq \mu < 1$ and $v > 0$. 
    
    Let $\vz_t=(1-p_t)\vx_t+p_t \vz_{t-1}$ for $t\geq 0$, with $p_0=1$, $\vz_{-1} = \vx^*$, and, for $t\ge1$, we define
    \begin{equation}\label{eq: p_def}
    p_t=\frac{\mu+T-t+1}{T-t+2}.
    \end{equation}
   
    Then, we have
    \begin{equation}\label{bigTheorem}
        \E\left[h(\vx_{T+1}) - h^* \right] \leq \frac{1}{\alpha_T}\left(h(\vx_0) - h^*+ \frac{1}{2 \tau } \|\vx_0 - \vx^*\|^2\right) +  \frac{v}{\alpha_T} \sum_{t=0}^{T} \alpha_t.
    \end{equation}
    where the sequence $(\alpha_t)$ is defined by $\alpha_{-1} = \alpha_0 = 1$ and
    \[
    \alpha_t
    =
    \frac{T - t + 2}{\mu + T - t + 1} \cdot \alpha_{t-1},
    \qquad t = 1,\dots,T.
    \]
\end{lemma}
\begin{proof}
Since $\mu < 1$, we have $p_t \in [0,1]$. Additionally, from the definition of $\alpha_t$, we have $\alpha_t p_t = \alpha_{t-1}$ for all $t \ge 0$. Consequently, $\alpha_t$ is non-decreasing.
    Multiplying both sides of \eqref{proxcomb111} by $\alpha_t$, we have
    \begin{align} \label{proxcomb112}
        \alpha_t \E[h(\vx_{t+1}) -  h(\vz_t) - \mu h(\vx_t) + \mu h^*] &\leq \frac{\alpha_t}{2\tau}  \E\|\vx_t - \vz_t\|^2 - \frac{\alpha_t}{2\tau}  \E\|\vx_{t+1} - \vz_t\|^2 + \alpha_t v.
    \end{align}
We estimate the first term on the right-hand side as
\[    \alpha_t\|\vx_t - \vz_t\|^2 = \alpha_t\|p_t \vx_t - p_t \vz_{t - 1}\|^2 = \alpha_tp_t^2 \|\vx_t - \vz_{t-1}\|^2 \leq \alpha_tp_t \|\vx_t - \vz_{t-1}\|^2=\alpha_{t-1}\|\vx_t-\vz_{t-1}\|^2,\] 
    where we used the definition of $\vz_t$, $p_t \leq 1$ and the definition of $\alpha_t$ which gives $\alpha_tp_t=\alpha_{t-1}$.
    
    Plugging this back into \eqref{proxcomb112} gives
    \begin{align} \label{proxcomb113}
        \alpha_t \E[h(\vx_{t+1}) -  h(\vz_t) - \mu h(\vx_t) + \mu h^*] &\leq \frac{\alpha_{t-1}}{2 \tau}  \E\|\vx_t - \vz_{t-1}\|^2 - \frac{\alpha_t}{2 \tau}  \E\|\vx_{t+1} - \vz_t\|^2 + \alpha_t v.
    \end{align}
    Notice that we can telescope the right-hand side by summing from $t=0$ to $T$. To complete this recursion, we define $\vz_{-1} = \vx^*$. Summing the inequality \eqref{proxcomb113} from $t = 0$ to $T$ and telescoping give
    \begin{align} 
        \underbrace{\sum_{t=0}^{T} \alpha_t \E[h(\vx_{t+1}) - h(\vz_t) - \mu h(\vx_t) + \mu h^*]}_{S}
        &\leq \frac{\alpha_{-1}}{2\tau}\|\vx_0 - \vz_{-1}\|^2
        + v\sum_{t=0}^{T}\alpha_t = \frac{1}{2\tau}\|\vx_0 - \vx^*\|^2 + v\sum_{t=0}^{T}\alpha_t \label{proxcomb114},
    \end{align}    
    where we used $\alpha_{-1}=1$. We shall proceed to obtain a lower bound of the left-hand side of \eqref{proxcomb114}. 

    Unrolling the recursive definition for $\vz_t$, we get:
    \begin{align}
        \vz_t 
        = (1-p_t)\vx_t + p_t \vz_{t-1} = \sum_{s=0}^t \left(\prod_{j=s+1}^t p_j\right) (1-p_s) \vx_s 
        + \left(\prod_{j=0}^t p_j\right) \vx^*, \label{proxztunroll}
    \end{align}
    where we used the convention $\prod_{j=t+1}^t p_j=1$.
    
    Using $p_j = \alpha_{j-1}/\alpha_j$, we have, for $0 \le s \le t$,
    \[
    \prod_{j=s+1}^t p_j = \prod_{j=s+1}^t \frac{\alpha_{j-1}}{\alpha_j}
    = \frac{\alpha_s}{\alpha_t},
    \qquad
    1 - p_s = 1 - \frac{\alpha_{s-1}}{\alpha_s}
    = \frac{\alpha_s - \alpha_{s-1}}{\alpha_s},\qquad
    \prod_{j=0}^t p_j = \prod_{j=0}^t \frac{\alpha_{j-1}}{\alpha_j}
    = \frac{\alpha_{-1}}{\alpha_t} = \frac{1}{\alpha_t}.
    \]
    Therefore \eqref{proxztunroll} becomes
    \begin{equation*}
        \vz_t
        = \sum_{s=0}^t 
        \frac{\alpha_s}{\alpha_t}
        \cdot
        \frac{\alpha_s - \alpha_{s-1}}{\alpha_s}
        \vx_s
        + \frac{1}{\alpha_t} \vx^*=  \sum_{s=0}^t \frac{\alpha_s - \alpha_{s-1}}{\alpha_t} \vx_s
        + \frac{1}{\alpha_t} \vx^*.
    \end{equation*}
    Notice that this is a convex combination of points $\vx_0, \dots, \vx_t, \vx^*$ as the weights sum to 1, that is,
    \begin{align*}
        \frac{1}{\alpha_t} + \sum_{s=0}^t \frac{\alpha_s - \alpha_{s-1}}{\alpha_t} = \frac{1}{\alpha_t}(1 + \alpha_t - \alpha_{-1}) = 1,
    \end{align*}
    since $\alpha_{-1}=1$ and each weight is between $[0,1]$ as $\alpha_t$ is non-decreasing. So this definition of $\vz_t$ agrees with our earlier assumption that $\vz_t \in \mathcal{F}(\vx_0, \dots, \vx_t, \vx^*)$. As $h$ is convex, by Jensen's inequality we have, 
    \begin{equation}\label{eq:hzjensen}
        h(\vz_t)
        \le
        \frac{1}{\alpha_t}
        \left(
        h(\vx^*) + \sum_{s=0}^t (\alpha_s - \alpha_{s-1}) h(\vx_s)
        \right).
    \end{equation} 
    Substituting \eqref{eq:hzjensen} into the left-hand side of the inequality \eqref{proxcomb114}
    \begin{align} 
        S
        &\geq
        \sum_{t=0}^{T} \alpha_t \E\left[
        h(\vx_{t+1})
        - \frac{1}{\alpha_t}
        \left(
        h^* + \sum_{s=0}^t (\alpha_s - \alpha_{s-1}) h(\vx_s)
        \right)
        - \mu h(\vx_t)
        + \mu h^*
        \right] \notag\\
        & = \sum_{t=0}^{T} \E\left[
        \alpha_t h(\vx_{t+1})
        - \sum_{s=0}^t (\alpha_s - \alpha_{s-1}) h(\vx_s)
        - \mu \alpha_t h(\vx_t)
        + (\mu\alpha_t-1)  h^*
        \right].
        \label{proxcomb115}
    \end{align}
    We next estimate the second term inside the bracket.
    \begin{align*}
        \sum_{s=0}^t (\alpha_s - \alpha_{s-1})\left(h(\vx_s) - h^* + h^*\right)
        &= \sum_{s=0}^t (\alpha_s - \alpha_{s-1})\left(h(\vx_s) - h^*\right)
        + h^*\left(\alpha_t - \alpha_{-1}\right) \\
        &= \sum_{s=0}^t (\alpha_s - \alpha_{s-1})\left(h(\vx_s) - h^*\right)  + \alpha_t h^* - h^*,
    \end{align*}
using $\alpha_{-1} = 1$.
Combining this with \eqref{proxcomb115} , we obtain a lower bound on the left-hand side of \eqref{proxcomb114}:
    \begin{align}
        S \geq
        \sum_{t=0}^{T} \E\left[
        \alpha_t (h(\vx_{t+1}) - h^*) - \mu \alpha_t(h(\vx_t) - h^*) - \sum_{s=0}^t (\alpha_s - \alpha_{s-1})\left[h(\vx_s) - h^*\right]
        \right]. \label{proxcomb117}
    \end{align}
    We next proceed to show that this lower bound gives us an expression for the last iterate. Define 
    \begin{equation}\label{eq: r_def}
    r_t = h(\vx_t) - h^*,
    \end{equation}
     then \eqref{proxcomb117} becomes
    \begin{align}
        S  
\geq \E\left[
        \sum_{t=0}^{T} \alpha_t r_{t+1}
        - \mu \sum_{t=0}^{T} \alpha_t r_t
        - \sum_{t=0}^{T} \sum_{s=0}^{t} (\alpha_s - \alpha_{s-1}) r_s
        \right].\label{eq: swr4}
    \end{align}
    For the double sum, we change the order of summation
    \[
    \sum_{t=0}^{T} \sum_{s=0}^{t} (\alpha_s - \alpha_{s-1}) r_s
    = \sum_{s=0}^{T} \sum_{t=s}^{T} (\alpha_s - \alpha_{s-1}) r_s
    = \sum_{s=0}^{T} (\alpha_s - \alpha_{s-1})(T - s + 1) r_s.
    \]
Using this in \eqref{eq: swr4} gives
    \begin{align}
S\geq        \E\left[
\sum_{t=0}^{T} \alpha_t r_{t+1}
        - \mu \sum_{t=0}^{T} \alpha_t r_t
        - \sum_{t=0}^{T} (\alpha_t - \alpha_{t-1})(T - t + 1) r_t
        \right].\label{eq: bcv4}
    \end{align}
We can rewrite the first term in the bracket as
    \begin{align*}
        \sum_{t=0}^{T} \alpha_t r_{t+1}
= \alpha_T r_{T+1} + \sum_{t=0}^{T-1} \alpha_t r_{t+1} = \alpha_T r_{T+1} + \sum_{t=1}^{T} \alpha_{t-1} r_t.
    \end{align*}
The second term inside the bracket in \eqref{eq: bcv4} can be rewritten as
    \[\mu \sum_{t=0}^{T} \alpha_t r_t = \mu \alpha_0 r_0 + \mu \sum_{t=1}^{T} \alpha_t r_t.\]
The third term inside the bracket in \eqref{eq: bcv4} can be rewritten as
    \begin{align*}
        \sum_{t=0}^{T} (\alpha_t - \alpha_{t-1})(T - t + 1) r_t
        &= (\alpha_0 - 1)(T+1) r_0 + \sum_{t=1}^{T} (\alpha_t - \alpha_{t-1})(T - t + 1) r_t.
    \end{align*}
    where in the last line we used $\alpha_{-1} = 1$.

    Combining the last three display equations in \eqref{eq: bcv4} gives
    \begin{align}\label{eq: sgf4mid}
S\geq  \E\left[
        \alpha_T r_{T+1}
        + \sum_{t=1}^{T} \alpha_{t-1} r_t
        - \mu \alpha_0 r_0 - \mu \sum_{t=1}^{T} \alpha_t r_t
        - (\alpha_0 - 1)(T+1) r_0
        - \sum_{t=1}^{T} (\alpha_t - \alpha_{t-1})(T - t + 1) r_t
        \right].
    \end{align}
    We now group the terms with $r_0$ and $r_t$. For terms with $r_0$, we have
    \[- \mu \alpha_0 r_0 - (\alpha_0 - 1)(T+1) r_0 = - r_0\left(\mu \alpha_0 + (\alpha_0 - 1)(T+1)\right).\]
    For terms with $r_t$ in \eqref{eq: sgf4mid}, we have
    \[\sum_{t=1}^{T} r_t\left(\alpha_{t-1} - \mu \alpha_t - (\alpha_t - \alpha_{t-1})(T - t + 1)\right) = \sum_{t=1}^{T} r_t\left(\alpha_{t-1}(T - t + 2) - \alpha_t\left(\mu + T - t + 1\right)\right).\]
    Plugging these into \eqref{eq: sgf4mid} gives 
    \begin{align}
S\geq        \E\left[
        \alpha_T r_{T+1}
        + \sum_{t=1}^{T} r_t\left(
        \alpha_{t-1}(T - t + 2)
        - \alpha_t\left(\mu + T - t + 1\right)
        \right)
        - r_0\left(\mu \alpha_0 + (\alpha_0 - 1)(T+1)\right)
        \right].\label{eq: sgf4}
    \end{align}
    Now notice that from our definition of $\alpha_t$, we have that for all $t \geq 1$,
    \[
    \alpha_{t-1}(T - t + 2)
    - \alpha_t\left(\mu + T - t + 1\right)
     = 0.\]
Using this in \eqref{eq: sgf4} gives
    \begin{align}\label{proxcomb118}
S\geq        \E\left[
        \alpha_T r_{T+1}
        - r_0\left(\mu \alpha_0 + (\alpha_0 - 1)(T+1)\right)
        \right]=\E\left[
        \alpha_T r_{T+1}
        - r_0 \mu
        \right],
    \end{align}
since $\alpha_0 = 1$ which gives $\mu \alpha_0 + (\alpha_0 - 1)(T+1) = \mu$. 
    Combining \eqref{proxcomb114} and \eqref{proxcomb118}, we have the inequality 
    \begin{align*} 
        \E\left[
        \alpha_T r_{T+1}
        \right] \leq \frac{1}{2 \tau} \|\vx_0 - \vx^*\|^2 + v \sum_{t=0}^{T} \alpha_t+r_0 \mu.
    \end{align*}
  Dividing by $\alpha_T$, using $\mu < 1$, and the definition of $r_t$ from \eqref{eq: r_def}, we obtain the assertion.
\end{proof}
\subsection{Technical lemmas needed for Theorem \ref{BIGTHEOREM}}\label{subsec: alpha-bds}
Now we proceed to obtain a concrete bound for polynomial step sizes where $\tau = \mathcal{O}(\frac{1}{\sqrt{T}})$. Before we do that, we first prove some technical lemmas that bound $\alpha_T$, $\sum_{t=0}^T \alpha_t$, and $T^\mu$. We use estimations similar to \citet{garrigos2025last}.
\begin{lemma}[Bounds on the $\alpha_t$ sequence]\label{technicallemma1}
    Let $T \ge 1$ and fix $\mu \in (0,1)$.
    Define $(\alpha_t)_{t=0}^{T}$ recursively by
    \[
    \alpha_0 = \alpha_{-1} = 1,
    \qquad
    \alpha_t
    = \frac{T - t + 2}{T - t + 1 + \mu}\alpha_{t-1},
    \quad t = 1, \dots,T.
    \]
    Then it holds that
    \begin{enumerate}
        \item $\alpha_T \ge \frac{(T+1)^{1-\mu}}{2^{1-\mu}},$
        \item $\frac{\sum_{t=1}^T \alpha_t}{\alpha_T} \leq 8 T^\mu \ln (T + 1),$
    \end{enumerate}
\end{lemma}
\begin{proof}
    By unrolling the recursion of $\alpha_t$, we have
    \[
    \alpha_t
    = \prod_{k=1}^{t}
    \frac{T - k + 2}{T - k + 1 + \mu}.
    \]
    Recall the gamma function, $\Gamma(n) = (n-1)!$. So we can write the numerator of $\alpha_t$ using the gamma function as,
    \[
    \prod_{k=1}^{t} (T - k + 2)
    = \frac{\Gamma(T+2)}{\Gamma(T-t+2)}.
    \]
    
    And similarly, we can rewrite the denominator with the gamma function as
    \[
    \prod_{k=1}^{t} (T - k + 1 + \mu)
    = \frac{\Gamma(T+\mu+1)}{\Gamma(T- t + \mu + 1)}.
    \]
    Combining the last two estimates, we can rewrite $\alpha_t$ as 
    \[
    \alpha_t
    =
    \frac{\Gamma(T+1 + 1)\Gamma(T-t+1 + \mu)}
    {\Gamma(T + 1 + \mu)\Gamma(T-t + 1 +1)}.
    \]
    
    By Gautschi's inequality \citep{gautschi1959gamma}, we have that
    for all $\vx > 0$ and all $c \in [0,1]$,
    \[
    \vx^{1-c}
    \le
    \frac{\Gamma(\vx+1)}{\Gamma(\vx+c)}
    \le
    (\vx+1)^{1-c}.
    \]
    Applying this for $\vx = T +1$ and $\vx = T-t+1$,
    \[
    (T+1)^{1-\mu}
    \le
    \frac{\Gamma(T+1+1)}{\Gamma(T+1+\mu)}
    \le
    (T+2)^{1-\mu},
    \]
    and similarly
    \[
    (T - t + 1)^{1-\mu}
    \le
    \frac{\Gamma(T - t + 1+1)}{\Gamma(T - t + 1 + \mu)}
    \le
    (T - t + 2)^{1-\mu}.
    \]
    So we have the lower bound for $\alpha_t$:
    \[
    \alpha_t
    \ge
    \frac{(T+1)^{1-\mu}}{(T - t + 2)^{1-\mu}}.
    \]
    In particular, for $t = T$, we have 
    \begin{equation}\label{eq: alpha_lb}
    \alpha_T \ge \frac{(T+1)^{1-\mu}}{2^{1-\mu}}.
    \end{equation}
    We also have an upper bound for $\alpha_t$ as
    \[
    \alpha_t
    \le
    \frac{(T+2)^{1-\mu}}{(T - t + 1)^{1 - \mu}}.
    \]
    Now, we sum this inequality over $t$ and upper bound the sum with the integral test, 
    \begin{align*}
        \sum_{t=1}^T \alpha_t & \leq \sum_{t=1}^T \frac{(T+2)^{1-\mu}}{(T - t + 1)^{1 - \mu}}\leq (T+2)^{1-\mu} \sum_{t=1}^T \frac{1}{(T - t + 1)^{1 - \mu}} \leq (T+2)^{1-\mu} \left( 1 + \frac{T^\mu - 1}{\mu} \right).
    \end{align*}
We use this bound and \eqref{eq: alpha_lb} to get
    \begin{align*}
        \frac{\sum_{t=1}^T \alpha_t}{\alpha_T} &\leq (T+2)^{1-\mu} \left( 1 + \frac{T^\mu - 1}{\mu} \right) \cdot \frac{2^{1-\mu}}{(T+1)^{1-\mu}} \\
        &= \left(\frac{2(T+2)}{T+1}\right)^{1-\mu} \left( 1 + \frac{T^\mu - 1}{\mu} \right) \\
        & \leq 4 \left(1 + \frac{T^\mu - 1}{\mu}\right),
    \end{align*}
    where in the last line we used the fact that $\frac{T+2}{T+1} \leq 2$ and $\mu < 1$. So we have the estimate for $\frac{\sum_{t=1}^T \alpha_t}{\alpha_T}$:
    \[
    \frac{\sum_{t=1}^T \alpha_t}{\alpha_T} \leq 4 \left(1 + \frac{T^\mu - 1}{\mu}\right).
    \]
    Finally, we bound the right-hand side.
    Define
    \[
    \psi(T)
    = 8T^{\mu}\ln\left(T+1\right)
    - 4\left(1 + \frac{T^{\mu}-1}{\mu}\right).
    \]
    We will show that $\psi(T)\ge 0$ for every $T\ge 1$. We compute the derivative of $\psi$:
    \[
    \psi'(T)
    = 8\mu T^{\mu-1}\ln\left(T+1\right)
    + \frac{8T^{\mu}}{T+1}
    - 4T^{\mu-1} = 4T^{\mu-1}
    \left(
    2\mu\ln\left(T+1\right)
    + \frac{2T}{T+1}
    - 1
    \right).
    \]
    Since $T\ge 1$, we have
    $\ln\left(T+1\right)\ge 0$ and $\frac{T-1}{T+1}\ge 0$, and hence,
    \[
    2a\ln\left(T+1\right)
    + \frac{2T}{T+1}
    - 1
    = 2\mu\ln\left(T+1\right) + \frac{T-1}{T+1}
    \ge 0,
    \]
    and thus $\psi'(T)\ge 0$ for every $T\ge 1$.
    Furthermore,
$    \psi(1)=8\ln\left(2\right)-4>0.$
    Hence, $\psi$ is nondecreasing and $\psi(1)>0$, this implies 
    $\psi(T)\ge 0$ for every $T\ge 1$. This gives us the assertion.
\end{proof}
We now show that $T^\mu$ is a constant given that $\tau$ has a polynomial dependence on $T$ and for the specific choice of $\mu = 2 \tau L (1+\gamma)$ we had obtained \Cref{proxobjectivecomb_main}.  
\begin{lemma}[$T^\mu$ for polynomially decaying step sizes]\label{boundedTphi}
    Let $T \ge 1$ and $C > 2$.  
If we have $\tau = \frac{1}{C L \sqrt{T}}$ and 
$    \mu = 2 \tau L \left(1 + \gamma\right)$,
    where $\gamma = \frac{1-2 \tau L}{1 + 2 \tau L}$, then it follows that
$    \tau L = \frac{1}{C \sqrt{T}} \le \frac{1}{C} < \frac{1}{2},$
    and
    \[
    T^\mu \leq \exp\left(\frac{8}{e C}\right).
    \]
\end{lemma}
\begin{proof}
    We use the definition of $\mu$ to derive
        \begin{align}
        T^\mu
        = T^{2\tau L(1+\gamma)} 
        = \exp\left(2\tau L(1+\gamma)\ln T\right) 
        = \exp\left(\frac{2(1+\gamma)\ln T}{C \sqrt{T}}\right).\label{eq: hdr4}
    \end{align}
    We now use the fact that $\frac{\ln T}{\sqrt{T}} \leq \frac{2}{e}$ for every $T \ge 1$ in \eqref{eq: hdr4} to get
    \begin{align*}
        T^\mu
        \le\exp\left(\frac{4(1+\gamma)}{e C}\right).
    \end{align*}
    Using our definition of $\gamma = \frac{1 - 2\tau L}{1 + 2\tau L}$, we can estimate 
$    1+\gamma 
    = \frac{2}{1+2\tau L} \leq 2,$
    since $1+2\tau L \ge 1$.  
    Substituting this back gives
the last assertion.
\end{proof}

\subsection{Proof of Theorem \ref{BIGTHEOREM}: Last-iterate analysis}\label{subsec: proof_main_thm_spgd}
We combine the above results to obtain the last iterate bound for step size $\tau = \frac{1}{CL \sqrt{T}}$ for $C > 2$.
\begin{proof}[Proof of Theorem \ref{BIGTHEOREM}]
    From \Cref{proxobjectivecomb_main} we obtained the following
    \begin{equation*}
            \mathbb{E}\left[
            h(\vx_{t+1})
            - h(\vz_t)
            - \mu h(\vx_t)
            + \mu h^*
            \right]\leq
            \frac{1}{2\tau}
            \mathbb{E}\left[
            \|\vx_t - \vz_t\|^2
            -
            \|\vx_{t+1} - \vz_t\|^2
            \right]
            + v,
    \end{equation*}
    where $\mu = 2 \tau L (1+\gamma)$ and $v = (1 + 1/\gamma)\sigma_*^2 \tau$ for any $\gamma > 0$. 
    In particular, we choose $\gamma = \frac{1 - 2 \tau L}{1 + 2 \tau L}$. Then, notice that $\mu < 1$ since $\tau L < \frac{1}{2}$. Now, we can apply \Cref{theorem41} to get the following bound on the last iterate:
    \begin{align}\label{lic1}
        \E\left[h(\vx_{T+1}) - h^* \right] \leq \frac{1}{2 \tau \alpha_T} \|\vx_0 - \vx^*\|^2 + \frac{1}{\alpha_T} (h(\vx_0) - h^*) + + \frac{v \alpha_0}{\alpha_T} + \frac{v}{\alpha_T} \sum_{t=1}^{T} \alpha_t .
    \end{align}
    Now, Lemma \ref{technicallemma1} tells us that
    \begin{align*}
        \alpha_T &\geq \frac{(T+1)^{1 - \mu}}{2^{1 - \mu}} \geq \frac{T^{1 - \mu}}{2} \quad \text{and} \quad
        \frac{\sum_{t=1}^{T} \alpha_t}{\alpha_T} \leq 8 T^\mu \ln (T + 1).
    \end{align*}
    Plugging these into \eqref{lic1}, we get
    \begin{align}
        \E\left[h(\vx_{T+1}) - h^* \right] &\leq 
         T^{\mu} \left[ \frac{\|\vx_0 - \vx^*\|^2}{\tau T} +  \frac{2}{T}\left(h(\vx_0) - h^*+ v\right) + 8 v \ln (T + 1) \right]\label{eq: sxm4}
    \end{align}
    From our step size, $\tau = \frac{1}{CL \sqrt{T}}$ we get $\tau L \leq \frac{1}{C}$ and hence,
$\frac{1}{1-2\tau L} \leq \frac{C}{C-2}.$
    From our choice of $\gamma = \frac{1 - 2 \tau L}{1 + 2 \tau L}$, we get
\[v = (1 + 1/\gamma)\sigma_*^2 \tau = \frac{2}{1 - 2 \tau L}\sigma_*^2 \tau \leq \frac{2 \sigma_*^2 C \tau}{C-2} =  \frac{2 \sigma_*^2}{(C-2)L\sqrt{T}}.\] 
    Substituting our step size $\tau$, the above bound and the bound for $T^\mu$ from \Cref{boundedTphi} into \eqref{eq: sxm4} gives the assertion.
\end{proof}

\section{Main results for Section \ref{sec: rand_inc_prox}}\label{app: ripm}
We now provide the main results for the \eqref{eq:intro-incrementalpgd} algorithm, which are presented with a constant step size depending on the horizon $T$. For the statement in the main text, that is, \Cref{BIGTHEOREM2_main}, we used $C=5$ and simplified the bounds.
\begin{theorem}[Last-iterate convergence, incremental proximal, generalized version of \Cref{BIGTHEOREM2_main}]\label{BIGTHEOREM2}
    Consider the \eqref{eq:intro-incrementalpgd}  algorithm with step size $\tau = \frac{1}{C L \sqrt{T}}$ for some $C > 4$. Let Assumptions \ref{assumptions} and \ref{assumptions2_main} hold. Then, the last iterate $\vx_{T+1}$ satisfies
    \begin{align*}
        \E[h(\vx_{T+1}) - h^*] &\leq \exp\left(\frac{8}{e C}\right) \left[\frac{C L \|\vx_0 - \vx^*\|^2}{\sqrt{T}} + \frac{2}{T} (h(\vx_0) - h^*) + \frac{4 \sigma_*^2 }{(C-2)LT^{1.5}} + \frac{16 \sigma_*^2  \ln (T+1)}{(C-2)L\sqrt{T}} \right. \\
        &\quad \left. + \frac{16 m^2 L_g^2}{CL T^{1.5}} + \frac{64 m^2 L_g^2 \ln (T + 1)}{CL \sqrt{T}}\right].
    \end{align*}
\end{theorem}

\subsection{One iteration lemma for \eqref{eq:intro-incrementalpgd} }\label{subsec: one-it-proof-incr}
We first prove the following lemma that is the corresponding result to \Cref{proxobjectivecomb_main}.
\begin{lemma}[Per-iteration descent, incremental proximal, see \Cref{blockproxlinearcomb_main}] \label{blockproxlinearcomb}
Let Assumptions \ref{assumptions} and \ref{assumptions2_main} hold; and let $(\vx_t)_{t=0}^T$ be generated by \eqref{eq:intro-incrementalpgd} with a $\tau$ satisfying $0 < \tau L < \frac{1}{4}$.
    Then for all $t= 0, \dots T$ and $\vz_t$ in $\mathcal{F}(\vx_0, ..., \vx_t)$ we have
    \begin{equation}\label{eq:blockproxbound}
        \begin{aligned}
            \mathbb{E}\left[
            h(\vx_{t+1})
            - h(\vz_t)
            - \mu h(\vx_t)
            + \mu h^*
            \right]
            &\leq
            \frac{1}{2\tau}
            \mathbb{E}\left[
            \|\vx_t - \vz_t\|^2
            -
            \|\vx_{t+1} - \vz_t\|^2
            \right]
            + \left(1 + \frac{1}{\gamma}\right)\sigma_*^2 \tau + 8 \tau m^2L_g^2,
        \end{aligned}
    \end{equation}
    where $\mu = 2 \tau L (1+ \gamma)$ for some $\gamma > 0$.
\end{lemma}
\begin{proof}
    We can apply \Cref{onestepid} here by using $mg_j$ instead of $g$ as defined in our proximal step. We obtain,
    \begin{align}\label{ineq1}
        \|\vx_{t+1} - \vz_t\|^2 - \|\vx_t - \vz_t\|^2 & \leq 2 \tau (f_{i_t}(\vz_t) - f_{i_t}(\vx_t)+f(\vx_t) - f(\vx_{t+1})) + 2 \tau m(g_{j_t}(\vz_t) - g_{j_t}(\vx_{t+1}))  \notag\\ 
        &\quad + \frac{\tau}{\delta_1} \| \nabla f_{i_t}(\vx_t) - \nabla f(\vx_t) \|^2 + (\tau \delta_1 + \tau L - 1) \|\vx_t - \vx_{t+1}\|^2.
    \end{align}
Here, taking conditional expectation, and using $\E\|X-\E X\|^2 \leq \E\|X\|^2$, we have
    \begin{align}
        \E_t\|\vx_{t+1} - \vz_t\|^2 - \|\vx_t - \vz_t\|^2 & \leq 
        2 \tau \E_t[f(\vz_t) - f(\vx_{t+1})] + 2\tau m\E_t(g_{j_t}(\vz_t) - g_{j_t}(\vx_{t+1})) \notag \\
        &\quad +\frac{\tau}{\delta_1} \E_t\| \nabla f_{i_t}(\vx_t)\|^2+ (\tau \delta_1 + \tau L - 1)\E_t\|\vx_t - \vx_{t+1}\|^2 .\label{ineq15}
    \end{align}
        Consider the term $2 \tau m(g_{j_t}(\vz_t) - g_{j_t}(\vx_{t+1}))$. We can rewrite this as:
    \begin{align*}
        2 \tau m(g_{j_t}(\vz_t) - g_{j_t}(\vx_{t+1})) & = 2 \tau m(g_{j_t}(\vz_t) - g_{j_t}(\vx_t)) + 2 \tau m(g_{j_t}(\vx_t) - g_{j_t}(\vx_{t+1})) \\
        &\leq 2 \tau m\left(g_{j_t}(\vz_t) - \frac{1}{m}g(\vx_{t+1})\right) + 2\tau m\left(\frac{1}{m}g(\vx_{t+1}) - g_{j_t}(\vx_t)\right)  + 2 \tau m L_g \|\vx_t - \vx_{t+1}\|,
    \end{align*}
    where we used the Lipschitzness of $g_j$.
    Taking conditional expectation of the above (using $\E_t[g_{j_t}(\vu)] = \frac{1}{m}g(\vu)$ where $\vu$ is measurable with respect to the conditioning of the expectation) and then total expectation, we have
    \begin{align} \label{gterm}
        2 \tau m \E[g_{j_t}(\vz_t) - g_{j_t}(\vx_{t+1})] & \leq 2 \tau  \E[g(\vz_t) - g(\vx_{t+1})] + 2 \tau  \E[g(\vx_{t+1}) - g(\vx_t)] + 2 \tau m L_g \E\|\vx_t - \vx_{t+1}\|.
    \end{align}
    As each $g_j$ is Lipschitz, notice that $g$ is also Lipschitz with the constant $m L_g$ since by triangle inequality, we have
    \begin{align*}
        |g(\vx) - g(\vy)| & = \left|\sum_{j=1}^m g_j(\vx) - \sum_{j=1}^m g_j(\vy)\right| \leq \sum_{j=1}^m |g_j(\vx) - g_j(\vy)|  \leq m L_g \|\vx-\vy\|.
    \end{align*}
Then, using the Lipschitzness of $g$ in \eqref{gterm}, we have
    \begin{align}\label{gterm2}
        2 \tau m\E[g_{j_t}(\vz_t) - g_{j_t}(\vx_{t+1})] & \leq 2 \tau \E[g(\vz_t) - g(\vx_{t+1})] + 4 \tau m L_g \E\|\vx_t - \vx_{t+1}\|.
    \end{align}
Next, for \eqref{ineq15}, we can use \Cref{variancetransfer} to bound the third term on the right-hand side,
 take total expectation and use \eqref{gterm2} to bound the second term on the right-hand side to deduce
    \begin{align}
        \E[\|\vx_{t+1} - \vz_t\|^2 - \|\vx_t - \vz_t\|^2] & \leq 2 \tau \E[f(\vz_t) - f(\vx_{t+1})+g(\vz_t) - g(\vx_{t+1})] + \frac{2\tau L(1+\gamma)}{\delta_1}\E[h(\vx_t) - h^*] \notag \\
        &\quad+ (\tau \delta_1 + \tau L - 1)\E\|\vx_t - \vx_{t+1}\|^2 + 4 \tau m L_g \E\|\vx_t - \vx_{t+1}\|+\frac{\tau(\gamma+1)}{\gamma\delta_1} \sigma_*^2.\label{eq: som4}
    \end{align}
    We define $\delta_1 = \frac{3}{4\tau} - L > 0$, which gives $1 - \tau \delta_1 - \tau L > 0$. Hence, we estimate the second from last term here by using Young's inequality to get
    \begin{equation*}
    4\tau mL_g \| \vx_t-\vx_{t+1}\| \leq (1-\tau\delta_1-\tau L) \| \vx_t -\vx_{t+1}\|^2 + \frac{4\tau^2m^2L_g^2}{1-\tau \delta_1 -\tau L}.
    \end{equation*}
Applying this estimate in \eqref{eq: som4} and rearranging, we have
    \begin{align*}
        2 \tau \E[h(\vx_{t+1}) - h(\vz_t)] & \leq \E[\|\vx_t - \vz_t\|^2 - \|\vx_{t+1} - \vz_t\|^2] + \frac{2\tau L(1+ \gamma)}{\delta_1}\E[h(\vx_t) - h^*] + \frac{4\tau^2 m^2 L_g^2}{1 - \tau \delta_1 - \tau L}+\frac{\tau(\gamma+1)}{\gamma\delta_1} \sigma_*^2.
    \end{align*}
    We next use $\frac{\tau}{\delta_1} = \frac{4\tau^2}{3-4\tau L} \leq 2 \tau^2$, which uses $\tau L < \frac{1}{4}$ and $3 - 4\tau L > 2$, divide both sides by $2\tau$ and rearrange to obtain
    \begin{align*}
        \E[h(\vx_{t+1}) - h(\vz_t) - 2\tau L (1+\gamma)h(\vx_t) + 2\tau L (1+\gamma)h^*] & \leq \frac{1}{2 \tau} \E[\|\vx_t - \vz_t\|^2 - \|\vx_{t+1} - \vz_t\|^2] \\
        &\quad + \left(1+ \frac{1}{\gamma}\right) \tau \sigma_*^2  + \frac{2\tau m^2 L_g^2}{1 - \tau \delta_1 - \tau L}.
    \end{align*}
    Finally, we utilize the equality $1 - \tau \delta_1 - \tau L = \frac{1}{4}$ for the last term and then
use the notation $\mu = 2\tau L (1+\gamma)$.
\end{proof}

\subsection{Proof of Theorem \ref{BIGTHEOREM2}: Last-iterate analysis of \eqref{eq:intro-incrementalpgd}}\label{subsec: incr_mainthm_proof}
\begin{proof}
    From \eqref{eq:blockproxbound}, we obtained
        \begin{equation}\label{eq:blockproxbound}
        \begin{aligned}
            \mathbb{E}\left[
            h(\vx_{t+1})
            - h(\vz_t)
            - \mu h(\vx_t)
            + \mu h^*
            \right]
            &\leq
            \frac{1}{2\tau}
            \mathbb{E}\left[
            \|\vx_t - \vz_t\|^2
            -
            \|\vx_{t+1} - \vz_t\|^2
            \right]
            + v,
        \end{aligned}
    \end{equation}
    where $\mu = 2 \tau L (1+ \gamma)$ and $v = (1 + 1/\gamma)\sigma_*^2 \tau + 8 \tau m^2L_g^2$ for some $\gamma > 0$. In particular, we choose $\gamma = \frac{1 - 2 \tau L}{1 + 2 \tau L}$, so that $\mu = \frac{4\tau L}{1 + 2 \tau L} < 1$ since $\tau < \frac{1}{4L}$. Hence, we can apply \Cref{theorem41} to obtain the following bound for the last iterate:
    \[\E\left[h(\vx_{T+1}) - h^* \right] \leq \frac{1}{2 \tau \alpha_T} \|\vx_0 - \vx^*\|^2 + \frac{v \alpha_0}{\alpha_T} + \frac{v}{\alpha_T} \sum_{t=1}^{T} \alpha_t + \frac{1}{\alpha_T}[h(\vx_0) - h^*].\]
    Define $v_a = (1 + 1/\gamma)\sigma_*^2 \tau$ to denote the first term in the expression for $v$. This is the same $v$ term obtained in \Cref{BIGTHEOREM}. Expanding $\mu$ and $v = v_a + 8 \tau m^2L_g^2$, we obtain
    \begin{align}\label{eq: rpimain33}
        \E[h(\vx_{T+1}) - h^*] &\leq \underbrace{\frac{1}{2 \tau \alpha_T} \|\vx_0 - \vx^*\|^2 + \frac{v_a \alpha_0}{\alpha_T} + \frac{v_a}{\alpha_T} \sum_{t=1}^{T} \alpha_t + \frac{1}{\alpha_T}\E [h(\vx_0) - h^*]}_1 + \underbrace{\frac{8 \tau m^2 L_g^2}{\alpha_T} \sum_{t=0}^T \alpha_t}_{2}.
    \end{align}
    It suffices to bound (2) in \eqref{eq: rpimain33} since we already have bounds for (1) from \Cref{BIGTHEOREM}. Splitting this term, we have
    \begin{align}\label{eq: rpim_23}
        \frac{8 \tau m^2 L_g^2}{\alpha_T} \sum_{t=0}^T \alpha_t &= 8 \tau m^2 L_g^2 \frac{\alpha_0}{\alpha_T} + 8 \tau m^2 L_g^2 \frac{\sum_{t=1}^T \alpha_t}{\alpha_T}.
    \end{align}
    Consider the first term on the right-hand side. Using $\alpha_0 = 1$, \Cref{technicallemma1}, and the definition of $\tau$, we have
    \begin{align*}
        8 \tau m^2 L_g^2 \frac{\alpha_0}{\alpha_T} &\leq 8 \tau m^2 L_g^2 \cdot \frac{2}{T^{1 - \mu}} = T^\mu \frac{16 m^2 L_g^2}{CL T^{1.5}}.
    \end{align*}
For the second term on the right-hand side of \eqref{eq: rpim_23}, using \Cref{technicallemma1} and the definition of $\tau$ gives
    \begin{align*}
        8 \tau m^2 L_g^2 \frac{\sum_{t=1}^T \alpha_t}{\alpha_T} &\leq T^\mu\frac{64 m^2 L_g^2 \ln (T + 1)}{CL \sqrt{T}}.
    \end{align*}
    Using \Cref{boundedTphi}, and combining the above two bounds, we have
    \begin{align}\label{eq: rpim_24}
        \frac{8 \tau m^2 L_g^2}{\alpha_T} \sum_{t=0}^T \alpha_t &\leq \exp\left(\frac{8}{eC}\right) \left(\frac{16 m^2 L_g^2}{CL T^{1.5}} + \frac{64 m^2 L_g^2 \ln (T + 1)}{CL \sqrt{T}}\right),
    \end{align}
   where we used the upper bound for $T^\mu$ from \Cref{boundedTphi}. Now plugging in the estimates for (1) from \Cref{BIGTHEOREM} equations \eqref{lic1} and \eqref{eq: sxm4}  and the estimates for (2) from \eqref{eq: rpim_24} into \eqref{eq: rpimain33} gives the assertion.
\end{proof}

\subsection{Proof of \Cref{sec: spp}}\label{app: spp_proof}
Before proving \Cref{cor: spp}, we would like to clarify that the proof can be simplified by staying in $x_t$ instead of moving to $x_{t+1}$. Additionally, the bound obtained below can be improved in terms of the non dominant terms by removing the $g(x_0) - g(x^*)$ term from the right-hand side as $\mu = 0$. 
However, our goal was to reuse the previous lemma to avoid repeating several calculations. A cleaner proof can be obtained by redoing the calculations for the case $\mu = 0$ and staying in $x_t$ instead of $x_{t+1}$. 

\begin{proof}[Proof of \Cref{cor: spp}]
    We begin from \eqref{eq:prox1}
    \begin{align*}
        \|\vx_{t+1} - \vz_t\|^2 - \|\vx_t - \vz_t\|^2 & = \|\vx_{t+1} - \vx_t\|^2 + 2 \langle \vx_{t+1} - \vx_t, \vx_t - \vz_t\rangle 
    \end{align*}
    Similar to \eqref{eq: spgd_term2} in the proof of \Cref{onestepid} we use the optimality condition of the proximal point step, to obtain
    \begin{align}\label{eq: SPP1}
        \|\vx_{t+1} - \vz_t\|^2 - \|\vx_t - \vz_t\|^2 & \leq - \|\vx_{t+1} - \vx_t\|^2 + 2 \tau (g_{j_t}(\vz_t) - g_{j_t}(\vx_{t+1})).
    \end{align}
    Now, we estimate the term $2 \tau [g_{j_t}(\vz_t) - g_{j_t}(\vx_{t+1})]$ as follows:
    \begin{align*}
        2 \tau(g_{j_t}(\vz_t) - g_{j_t}(\vx_{t+1})) & = 2 \tau [g_{j_t}(\vz_t) - g(\vx_{t+1})] + 2 \tau [g(\vx_{t+1}) - g_{j_t}(\vx_{t})] + 2 \tau [g_{j_t}(\vx_{t}) - g_{j_t}(\vx_{t+1})] 
    \end{align*}
    We take conditional expectation, and then total expectation, and use Lipschitzness of $g_j$ to get
    \begin{align*}
        2 \tau \E[g_{j_t}(\vz_t) - g_{j_t}(\vx_{t+1})] & \leq 2 \tau \E[g(\vz_t) - g(\vx_{t+1})] + 2 \tau \E[g(\vx_{t+1}) - g(\vx_{t})] + 2 \tau L_g \E\|\vx_t - \vx_{t+1}\| \\
        & \leq 2 \tau \E[g(\vz_t) - g(\vx_{t+1})] + 4 \tau L_g \E\|\vx_t - \vx_{t+1}\|,
    \end{align*}
    where in the last line we used the Lipschitzness of $g$ to estimate $\E[g(\vx_{t+1}) - g(\vx_{t})]$. Plugging this estimate back into \eqref{eq: SPP1} after taking total expectation, we have
    \begin{align*}
        \E[\|\vx_{t+1} - \vz_t\|^2 - \|\vx_t - \vz_t\|^2] & \leq 2 \tau \E[g(\vz_t) - g(\vx_{t+1})] + 4 \tau L_g \E\|\vx_t - \vx_{t+1}\| - \E\|\vx_{t+1} - \vx_t\|^2.
    \end{align*}
    Using Young's inequality, we can estimate the second term on the right-hand side here as
    \begin{align*}
        4 \tau L_g \E\|\vx_t - \vx_{t+1}\| & \leq 4 \tau^2 L_g^2 +  \E\|\vx_t - \vx_{t+1}\|^2.
    \end{align*}
    Substituting this into the previous inequality, rearrange and divide by $2\tau$ to deduce 
    \begin{align*}
        \E[g(\vx_{t+1}) - g(\vz_t)] & \leq \frac{1}{2 \tau} \E[\|\vx_t - \vz_t\|^2 - \|\vx_{t+1} - \vz_t\|^2] + 2 \tau L_g^2.
    \end{align*}
    Now, we can apply \Cref{theorem41} with $\mu = 0$ and $v = 2 \tau L_g^2$ to obtain the following bound for the last iterate:
    \begin{align}\label{eq: SPP2}
        \E[g(\vx_{T+1}) - g^*] & \leq \frac{1}{\alpha_T} \left( g(\vx_0) - g^* + \frac{1}{2 \tau} \|\vx_0 - \vx^*\|^2 \right) + \frac{v}{\alpha_T} \sum_{t=0}^T \alpha_t,
    \end{align}
    where 
    \[\alpha_t = \frac{T - t + 2}{T - t + 1} \alpha_{t-1},\]
    with $\alpha_0 = \alpha_{-1} = 1$. Now, unrolling the recursion for $\alpha_t$ gives
    \begin{align*}
        \alpha_t & = \prod_{j=1}^t \frac{T - j + 2}{T - j + 1} \alpha_0 = \frac{T+1}{T - t + 1}.
    \end{align*}
    We estimate $\frac{1}{\alpha_T}$ as
    \[\frac{1}{\alpha_T} = \frac{T - T + 1}{T+1} = \frac{1}{T+1}.\]
We next estimate $\frac{1}{\alpha_T} \sum_{t=0}^T \alpha_t$ as
    \begin{align*}
        \frac{1}{\alpha_T} \sum_{t=0}^T \alpha_t & = \frac{1}{T+1} \sum_{t=0}^T \frac{T+1}{T - t + 1} = \sum_{t=0}^T \frac{1}{T - t + 1} \leq 1 + \ln (T+1).
    \end{align*}
    Plugging these estimates back into \eqref{eq: SPP2}, we obtain
    \begin{align*}
        \E[g(\vx_{T+1}) - g^*] & \leq \frac{1}{T+1} \left( g(\vx_0) - g^* + \frac{1}{2 \tau} \|\vx_0 - \vx^*\|^2 \right) + 2 \tau L_g^2 (1 + \ln (T+1)).
    \end{align*}
    Now, we use the definition of $\tau = \frac{1}{\sqrt{T+1}}$ to obtain the final bound for the last iterate:
    \begin{align*}
        \E[g(\vx_{T+1}) - g^*] & \leq \frac{g(\vx_0) - g^*}{T+1} + \frac{\|\vx_0 - \vx^*\|^2}{2 \sqrt{T+1}} + 2 L_g^2 \frac{1 + \ln (T+1)}{\sqrt{T+1}}.
    \end{align*}
We obtain the given bound by absorbing the constants.
\end{proof}

\subsection{Proof for Section \ref{sec: blockprox}}\label{app: blockprox}
\begin{proof}[Proof of \Cref{th: blockprox}]
By using a constant step size $\tau$ in Lemma~5.3 of \citet{LinKuangAlacaogluFriedlander2025} we have the inequality
\[
\E_t \|\vx_{t+1} - \vu\|^2 \leq \|\vy_t - \vu\|^2 - 2 \tau [g(\vx_t) - g(\vu)] + 2mL_g \tau^2 \|\nabla f(\vx_t)\| + 2m^2L_g^2 \tau^2
\]
for all $\vu$ measurable with respect to $\mathcal{F}_t$. Setting $\vu = \vz_t$ gives
\begin{align}\label{blockprox:main1}
    \E_t \|\vx_{t+1} - \vz_t\|^2 \leq \|\vy_t - \vz_t\|^2 - 2 \tau [g(\vx_t) - g(\vz_t)] + 2mL_g \tau^2 \|\nabla f(\vx_t)\| + 2m^2L_g^2 \tau^2.
\end{align}
Using Young's inequality, we can estimate the third term on the right-hand side as
\begin{align} \label{blockprox:youngs1}
    2mL_g \|\nabla f(\vx_t)\| \leq \|\nabla f(\vx_t)\|^2 + m^2L_g^2.
\end{align}

Using the definition of $\vy_t = \vx_t - \tau \nabla f(\vx_t)$, we expand
\begin{align}\label{blockprox:Fcvx}
\|\vy_t - \vz_t\|^2
&= \|\vx_t - \vz_t\|^2 - 2\tau \langle \nabla f(\vx_t), \vx_t - \vz_t\rangle + \tau^2\|\nabla f(\vx_t)\|^2 \notag\\
&\leq \|\vx_t - \vz_t\|^2 - 2\tau (f(\vx_t) - f(\vz_t)) + \tau^2\|\nabla f(\vx_t)\|^2 ,
\end{align}

where in the last line we used the convexity of $f$. Combining \eqref{blockprox:main1}, \eqref{blockprox:youngs1} and \eqref{blockprox:Fcvx} yields
\[
\E_t \|\vx_{t+1} - \vz_t\|^2
\leq \|\vx_t - \vz_t\|^2
- 2\tau[h(\vx_t) - h(\vz_t)]
+ 2\tau^2 \|\nabla f(\vx_t)\|^2
+ 3m^2L_g^2 \tau^2.
\]
By \citet[Lemma 5.7]{LinKuangAlacaogluFriedlander2025} or \Cref{variancetransfer}, we obtain
\[
\|\nabla f(\vx_t)\|^2 \leq (1+ \gamma)2L[h(\vx_t)-h^*] + \left(1 + \frac{1}{\gamma}\right)\|\nabla f(\vx^*)\|^2,
\]
that is true for any $\gamma>0$.

Hence, we have the inequality
\begin{align}\label{eq: blockproxineq3}
\E_t \|\vx_{t+1} - \vz_t\|^2
& \leq \|\vx_t - \vz_t\|^2
- 2\tau[h(\vx_t) - h(\vz_t)]
+ 4L(1+\gamma)\tau^2[h(\vx_t)-h^*]\notag\\
&\quad
+ 2\left(1 + \frac{1}{\gamma}\right)\tau^2\|\nabla f(\vx^*)\|^2
+ 3m^2L_g^2 \tau^2.
\end{align}
Rearranging and dividing both sides of the above inequality by $2 \tau$ gives
\begin{align*}
&\left(1 - 2L(1+\gamma)\tau\right) h(\vx_t)
- h(\vz_t)
+ 2L(1+\gamma)\tau\, h^* \\
&\quad \leq \frac{1}{2\tau}\|\vx_t - \vz_t\|^2
- \frac{1}{2\tau}\E_t \|\vx_{t+1} - \vz_t\|^2
+ \left(1 + \frac{1}{\gamma}\right)\tau\|\nabla f(\vx^*)\|^2
+ 2m^2L_g^2 \tau.
\end{align*}
In the last term, we replaced $\frac{3}{2}$ with $2$ since $m^2L_g^2 \tau > 0$. Now, by the optimality condition for $h = f + g$, we have $-\nabla f(\vx^*) \in \partial g(\vx^*).$
Moreover, since each $g_j$ is $L_g$-Lipschitz, their sum $g$ is $mL_g$-Lipschitz, hence every subgradient of $g$ has norm at most $mL_g$. Thus, we have 
\[\|\nabla f(\vx^*)\| \leq mL_g.\] 

Using this in \cref{eq: blockproxineq3} gives
\begin{align*}
&\left(1 - 2\tau L(1+\gamma)\right) h(\vx_t)
    - h(\vz_t)
    + 2\tau L(1+\gamma) h^* \\
    &\quad \leq \frac{1}{2\tau}\|\vx_t - \vz_t\|^2
    - \frac{1}{2\tau}\E_t \|\vx_{t+1} - \vz_t\|^2
    + \left(3 + \frac{1}{\gamma}\right)m^2L_g^2 \tau .
\end{align*}

By taking $\mu = 1 - 2\tau L(1+\gamma)$, $b = -1$, $c = 2\tau L(1+\gamma)$ and $v = \left(3 + \frac{1}{\gamma}\right)m^2L_g^2 \tau$, we obtain
\begin{align}\label{blockprox:finalrecursion}
    \mu h(\vx_t) + b h(\vz_t) + c h^* \leq \frac{1}{2\tau}\|\vx_t - \vz_t\|^2
    - \frac{1}{2\tau}\E_t \|\vx_{t+1} - \vz_t\|^2 + v.
\end{align}
This expression has the recursive structure of \citet[Lemma 4.2]{garrigos2025last} with different coefficients. To conclude, we verify the conditions of \citet[Lemma 4.3]{garrigos2025last}:
\begin{itemize}
    \item $\mu + b + c = 0$: holds by construction.
    \item $\mu > 0$: holds when $\tau L < \frac{1}{2}$ (which is assumed) and by chosing $\gamma = (1 - 2\tau L)/(1 + 2\tau L) > 0$. This choice gives us:
    \[\mu = \frac{1-2\tau L}{1 + 2 \tau L}, \quad b = -1, \quad c = \frac{4\tau L}{1 + 2 \tau L}, \quad v = \frac{4 - 4\tau L}{1 - 2\tau L}m^2L_g^2\tau < \frac{4}{1 - 2\tau L}m^2L_g^2 \tau.\]
    \item $b \leq 0$: holds.
\end{itemize}
Additionally, notice that the term $v$ in our case is identical to the one in \citet[Lemma 4.2]{garrigos2025last} up to the constant coefficient ($4 m^2 L_g^2$ in ours vs $\sigma_*^2$ in theirs). Hence, we can apply \citet[Theorem 3.1]{garrigos2025last} and \citet[Corollary 3.4]{garrigos2025last} to obtain the 
$\mathcal{O}\left(\frac{\ln (T+1)}{\sqrt{T}}\right)$ rate.
\end{proof}

\section{Additional experimental results} \label{app: additionalexp}
Below, we provide additional experimental results for the synthetic lasso problem using BlockProx. Additional details regarding experimental setup can be found in \Cref{sec: experiments} in the main text.
\begin{figure}[H]
    \centering
    \begin{subfigure}[t]{0.33\linewidth}
        \centering
        \includegraphics[width=\linewidth]{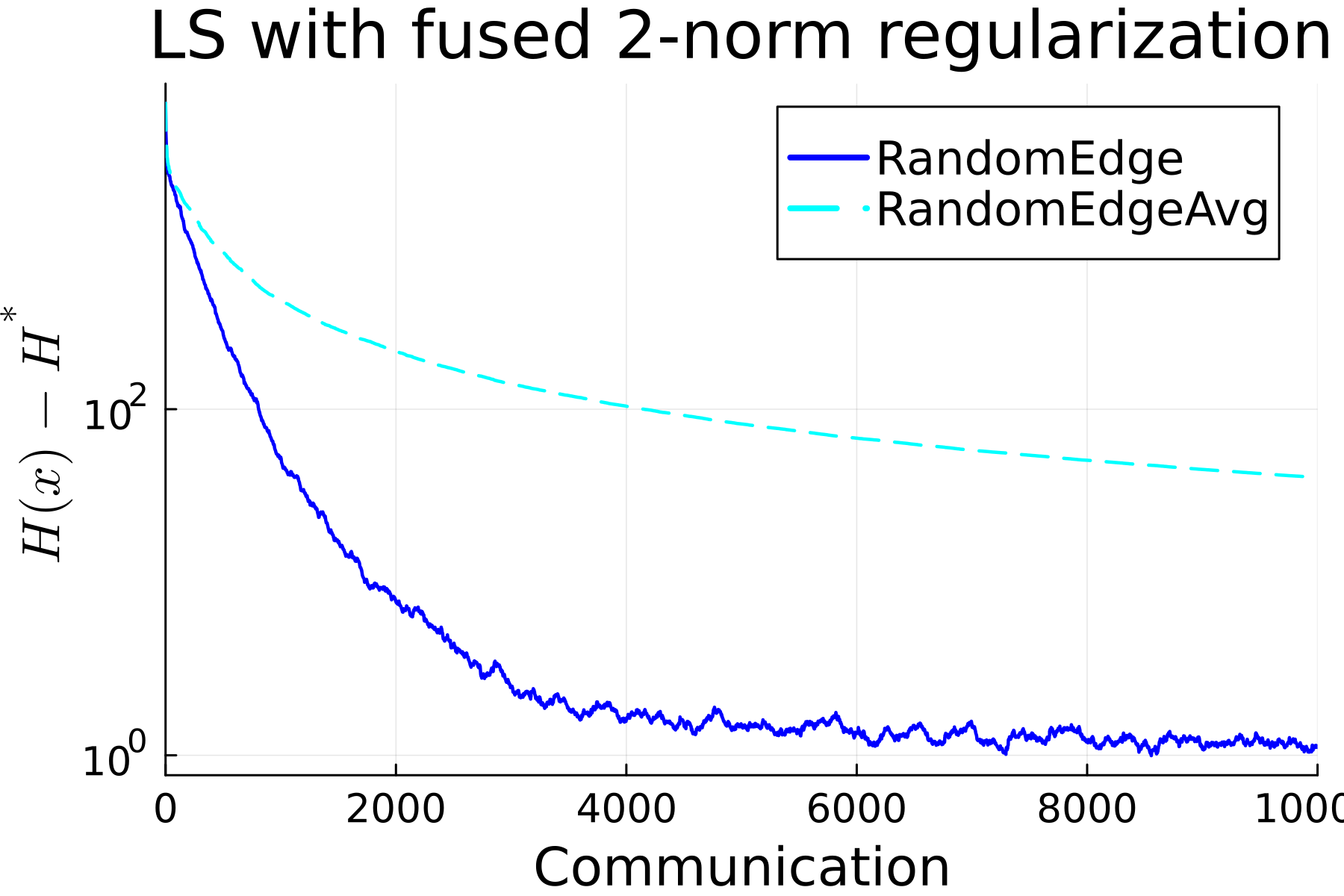}
        \caption{Synthetic lasso: 1 group, 40 nodes, $p=2$.}
        \label{fig:BlockProx_LS_2_1_40}
    \end{subfigure}
    \begin{subfigure}[t]{0.33\linewidth}
        \centering
        \includegraphics[width=\linewidth]{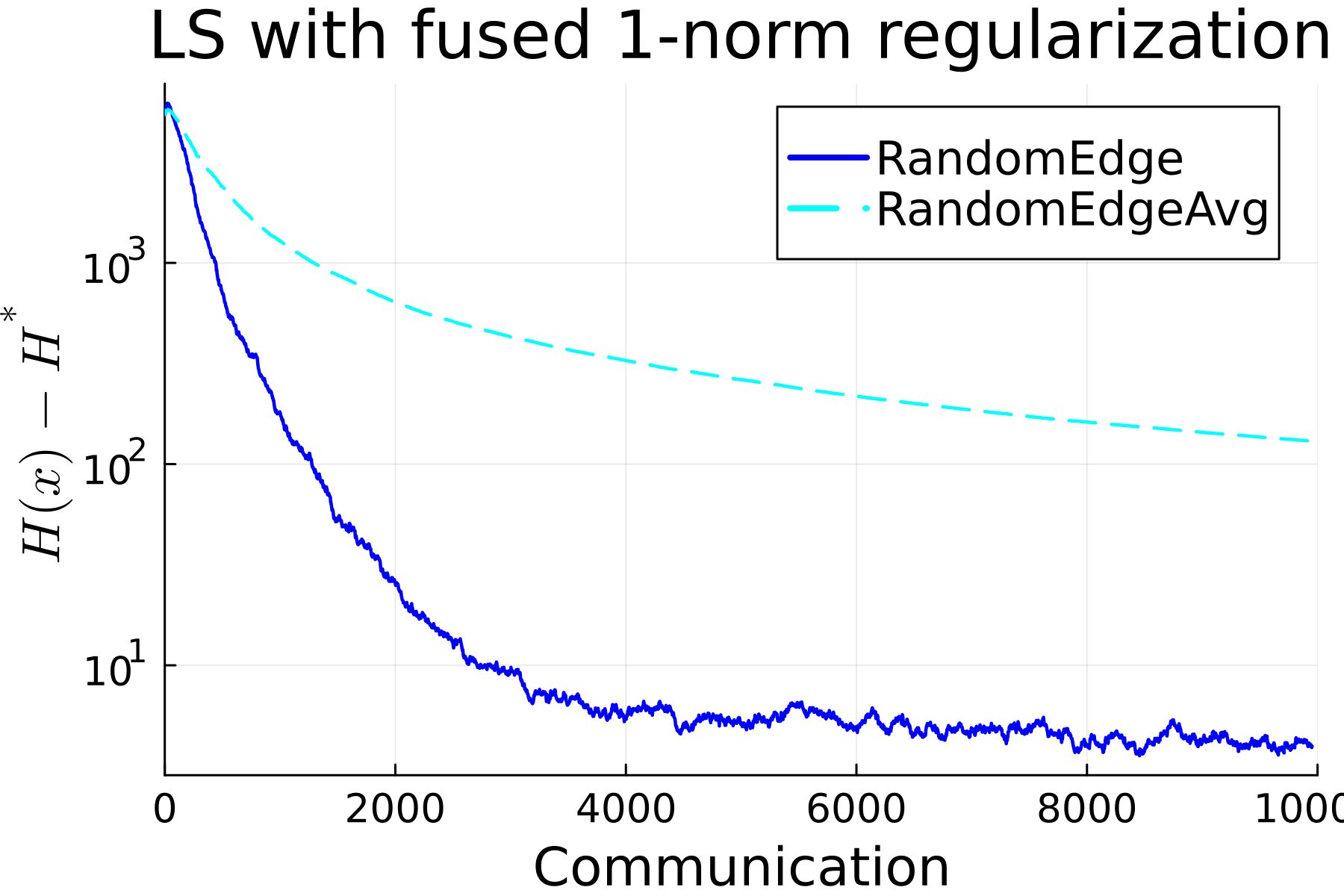}
        \caption{Synthetic lasso: 1 group, 40 nodes, $p=1$.}
        \label{fig:BlockProx_LS_1_1_40}
    \end{subfigure}
    \par\medskip
    \begin{subfigure}[t]{0.33\linewidth}
        \centering
        \includegraphics[width=\linewidth]{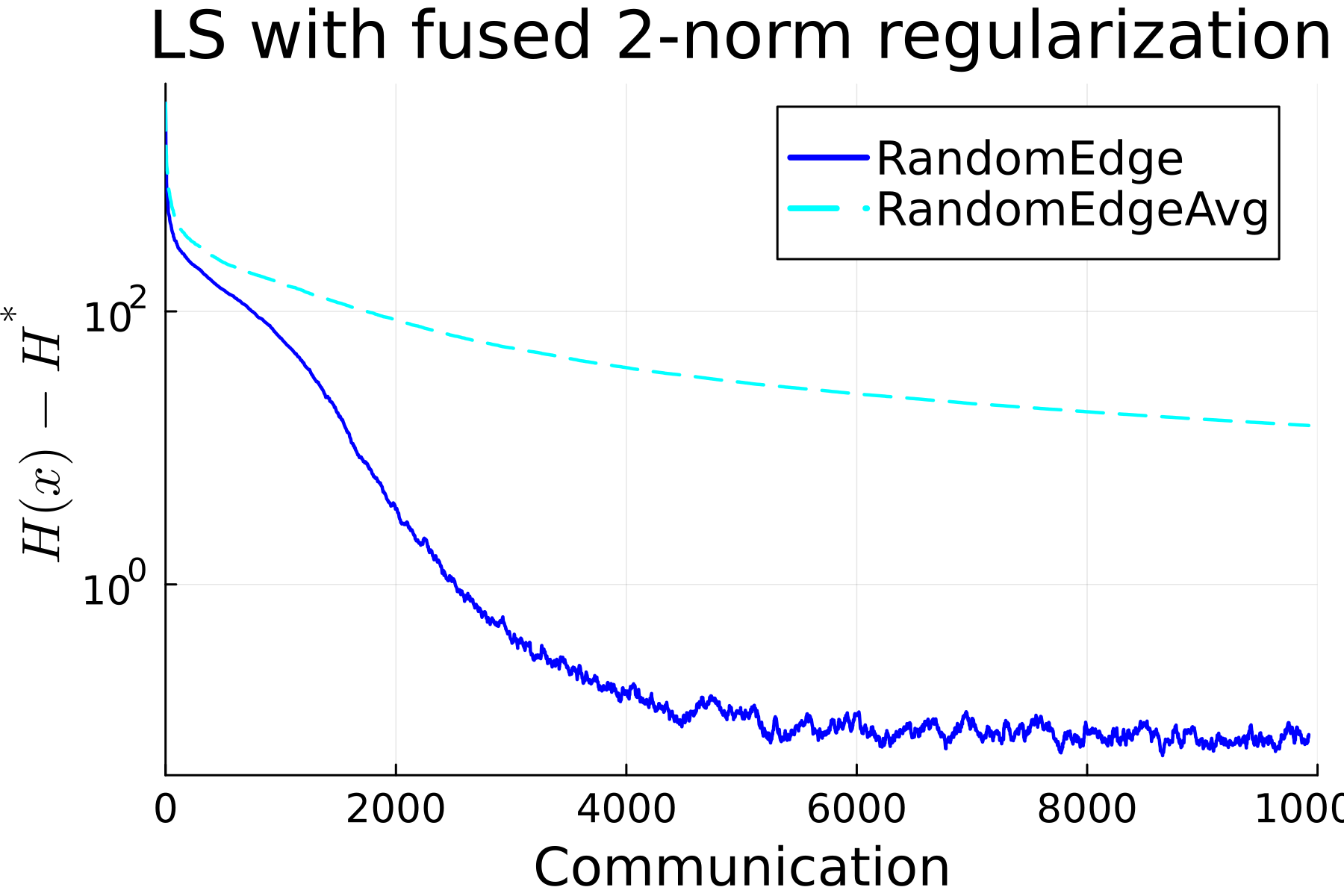}
        \caption{Synthetic lasso: 1 group, 20 nodes, $p=2$.}
        \label{fig:BlockProx_LS_2_1_20}
    \end{subfigure}
    \begin{subfigure}[t]{0.33\linewidth}
        \centering
        \includegraphics[width=\linewidth]{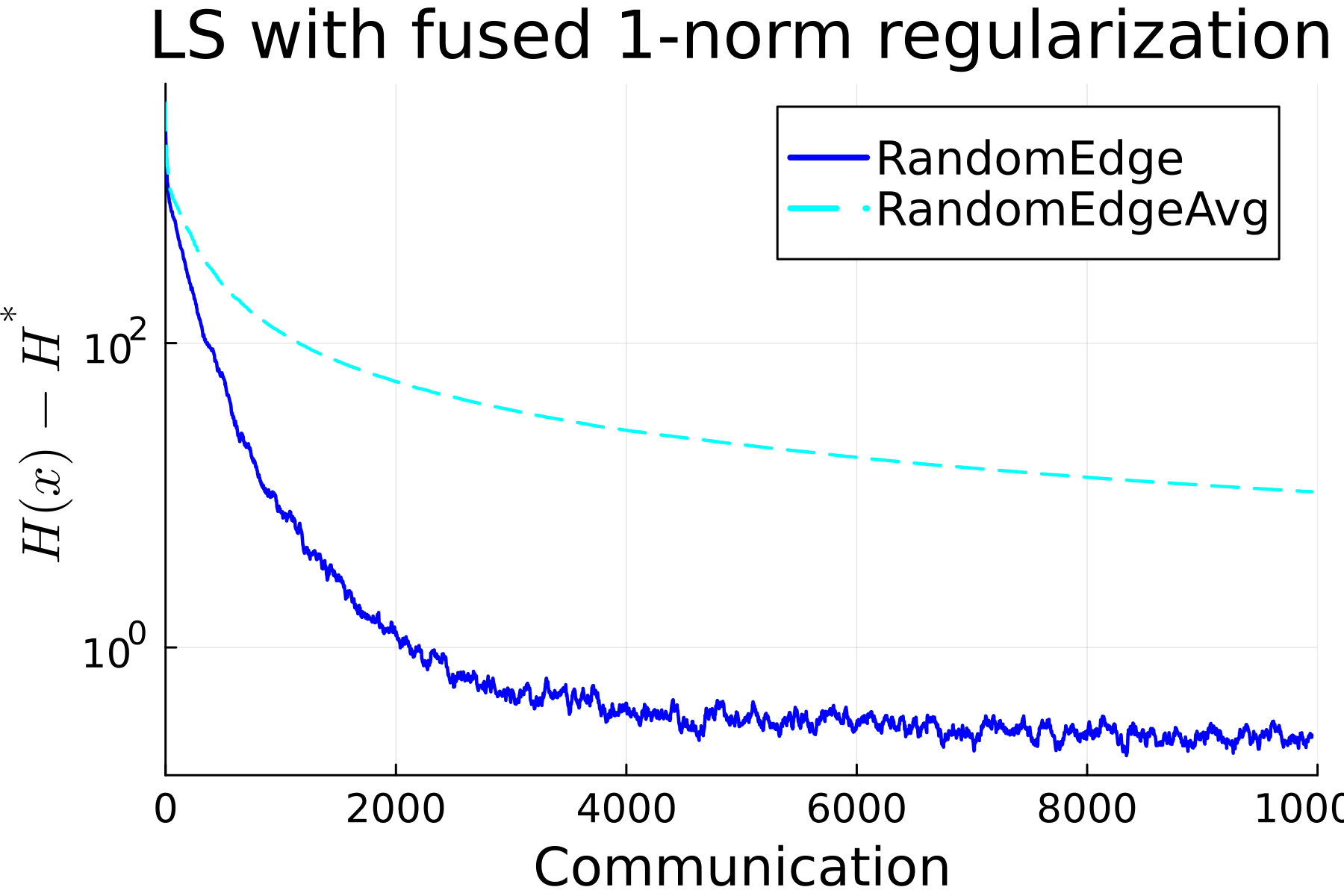}
        \caption{Synthetic lasso: 1 group, 20 nodes, $p=1$.}
        \label{fig:BlockProx_LS_1_1_20}
    \end{subfigure}
    \caption{Additional results for the synthetic lasso problem using BlockProx.}
    \label{fig:BlockProx_LS_Addn}
\end{figure}

\end{document}